\documentclass[12pt]{amsart}
\usepackage[left=2cm,top=2cm,right=3cm,nofoot]{geometry}
\usepackage{amsmath}%
\usepackage{amsfonts}%
\usepackage{amssymb}%
\usepackage[pdftex]{graphicx}
\newtheorem{thm}{Theorem}[section]
\theoremstyle{plain}

\newtheorem{corollary}[thm]{Corollary}

\newtheorem{definition}[thm]{Definition}

\newtheorem{lemma}[thm]{Lemma}

\newtheorem{proposition}[thm]{Proposition}
\newtheorem{remark}[thm]{Remark}

\newcommand{\R}{\mathbb{R}}
\newcommand{\Rn}{\mathbb{R}^{n}}
\newcommand{\pa}{\partial}

\newcommand{\Om}{\Omega}
\numberwithin{equation}{section}
\newcommand{\vt}{\vartheta}
\newcommand{\ve}{\varepsilon}
\newcommand{\vf}{\varphi}
\begin{document}
\title[Dirichlet Problem]{Estimates of the Green function  and the initial-Dirichlet problem for the heat equation in sub-Riemannian spaces}

\author{Nicola Garofalo}
\address
{Dipartimento di Ingegneria Civile, Edile e Ambientale (DICEA), Universit\`a di Padova, via Trieste 63, 35131
Padova, ITALY}%
\email[Nicola Garofalo]{nicola.garofalo@unipd.it}%

\thanks{First author supported in part by a grant ``Progetti d'Ateneo, 2014,'' University of Padova.}

\author{Isidro H. Munive}
\address{Center for Mathematical Research, CIMAT AC, Guanajuato, 36000 Guanajuato, Mexico}
\email{imunive@cimat.mx}%

\dedicatory{Dedicated to the memory of Gene Fabes}

\begin{abstract}
In a cylinder $D_T = \Om \times (0,T)$, where $\Om\subset \Rn$, we examine the relation between  the $L$-caloric 
measure,  $d\omega^{(x,t)}$, where $L$ is the heat operator associated with a system of  vector fields of H\"ormander type, and the measure $d\sigma_X\times dt$, where  $d\sigma_X$ is the intrinsic $X$-perimeter measure. The latter constitutes the appropriate replacement for the standard surface measure on the boundary
and plays a central role in sub-Riemannian geometric measure theory. Under suitable assumptions on the domain $\Om$ we establish the mutual absolute continuity of $d\omega^{(x,t)}$ and $d\sigma_X\times dt$. We also derive the solvability of the initial-Dirichlet problem for $L$ with boundary data in appropriate $ L^p$ spaces, for every $p>1$. 
\end{abstract}
\maketitle

\section{Introduction}
In this paper we study the initial-Dirichlet problem in a space-time cylinder $D_T = \Om \times (0,T)$ for the heat equation associated with a system $X=\left\{X_1,\ldots,X_m\right\}$ of $C^{\infty}$ vector fields in $\Rn$ satisfying H\"ormander's finite rank condition
\begin{equation}\label{frc}
\text{rank Lie}\ [X_1,\ldots,X_m] \equiv n.
\end{equation}
The heat operator in $\R^{n+1}$ that we are interested in is defined as follows
\begin{equation}
\label{HeatOp}
L = -\sum^m_{j=1}X^\star_jX_j-\frac{\pa}{\pa t},
\end{equation}
where $X^\star_j$ denotes the formal adjoint of $X_j$. Notice that the time-independent part of $L$ is in divergence form.
The results that we establish should be seen as the sub-Riemannian counterpart of those obtained by Fabes and Salsa in \cite{FSa} for the classical heat equation. In that paper the authors proved the unique solvability of the initial-Dirichlet problem in a cylindrical domain with Lipschitz cross-section, lateral data in $L^p$, $p > 2$, and zero initial values. They also showed that there is a Poisson kernel with the property that its $L^2$-averages over parabolic rectangles are equivalent to $L^1$-averages over the same sets.

In the non-Euclidean context of the present paper the class of Lipschitz domains is in essence meaningless, due to the generic presence of characteristic points on the lateral boundary of the cross-section $\Om$, i.e., points at which the system $X = \{X_1,...,X_m\}$ becomes tangent to the boundary $\pa \Om$. At such points a host of negative phenomena occur, see the discussion below and also the papers \cite{CG1} and \cite{CGN4}. Instead, we work with cylinders in which the cross-section $\Om$ satisfies some natural assumptions that, in the time independent case, allow to successfully solve the Dirichlet problem in $L^p$, for every $p>1$, see \cite{CGN2}, \cite{CGN3} and \cite{CGN4}.  For such class of cylinders, we prove the corresponding solvability of the parabolic initial-Dirichlet problem for boundary data in $L^p$, for every $p>1$, see Theorems \ref{HolderK} and \ref{HolderP} below.

Notice that when $X = \{\pa_{x_1},...,\pa_{x_n}\}$, then \eqref{HeatOp} reduces to the standard heat  operator $L = \Delta - \frac{\pa}{\pa t}$. A genuinely non-Euclidean situation of prototypical interest for the results in this paper is when the system $X$ is a basis of the bracket generating layer of a Carnot group $\mathbb G$ of step $r>1$. In such case we have $X_j^\star = - X_j$, and thus \eqref{HeatOp} becomes the heat operator $L = \sum^m_{j=1}X_j^2 -\frac{\pa}{\pa t}$ on $\mathbb G\times \R$, associated with the sub-Laplacian $ \sum^m_{j=1}X_j^2 $, see \cite{F2}. In the general case, we have $X_j^\star = - X_j + a_j$, where $a_j = - \operatorname{div} X_j$, and we can thus write 
\[
L  = \sum^m_{j=1} X_j^2 + X_0,
\]
where $X_0 = -\frac{\pa}{\pa t}- \sum_{j=1}^m a_j X_j $. Notice that in this situation we can write $X_0^\star = - X_0 + a_0$, with $a_0 = \sum_{j=1}^m \left(X_j a_j - a_j^2\right)$. We remark that, under the assumption \eqref{frc}, the system $\{X_0,X_1,...,X_m\}$ satisfies the finite rank condition with dimension $n+1$, i.e., the Lie algebra generates the whole of $\R^{n+1}$, see the introduction of \cite{H}.

For a domain $D\subset\R^{n+1}$, we let $\pa_{p}D$ denote the parabolic boundary of $D$; i.e., $\pa_p D$ is the set of points on the topological boundary of $D$ that can be connected to some interior point of $D$ by a closed curve having strictly increasing $t$-coordinate. 
In this work we focus on cylindrical  domains of the form 
\begin{equation}
\label{OmegaSet}
D_T=\Omega \times(0,T), 
\end{equation}
where $\Omega\subset\mathbb{R}^n$ is a connected bounded open set, and $T>0$.
With $S_T = \partial \Omega \times (0,T)$, it is well-known that the parabolic boundary of such $D_T$ is given by
\[
\partial_p D_T = S_T \cup (\Omega \times \{0\}).
\]
We denote by $L^* = -\sum^m_{j=1}X^*_jX_j + \frac{\pa}{\pa t}$ the formal adjoint of \eqref{HeatOp}. For a domain $D\subset\R^{n+1}$, the relevant parabolic boundary associated with $L^*$, denoted by  $\pa^*_p D$, contains the points on the boundary of $D$ that can be connected to some interior point of $D$ by a closed curve having strictly decreasing $t$-coordinate. The operator $L^*$ can be seen as the \emph{backwards heat operator}.

Thanks to H\"ormander's theorem \cite{H}, the assumption (\ref{frc}) guarantees the hypoellipticity of $L, L^*$.  The fundamental solution of $L$ with singularity at $(\xi,0)$, that we will denote by $p(x,\xi,t)$, exists,  it  is smooth and symmetric for $t>0$, and satisfies $L p(\cdot,\xi,\cdot) = 0$ for $t>0$.  Corresponding Gaussian estimates for the
heat kernel were independently obtained by Jerison and Sanchez-Calle \cite{JSC} and by Kusuoka
and Stroock \cite{KS}, see Theorem \ref{GaussB} below. One should also see the references \cite{VSC} and \cite{BBLU2}.

In $D_T$ such as \eqref{OmegaSet} above we consider the initial-Dirichlet problem 
\begin{equation}
\left\{
\begin{array}{ll}
Lu=0 & \mbox{in $D_T$},\\
u= \varphi & \mbox{on $\partial_pD_T$}.
\end{array} \right.
\label{Dirichlet-prob2}
\end{equation}
Using Bony's maximum principle \cite{Bony} one can show that for any $\varphi\in C(\partial_p D_T)$ there exists a
unique Perron-Wiener-Brelot-Bauer solution $H^{D_T}_{\varphi}$ to (\ref{Dirichlet-prob2}). If we fix a point $(x,t)\in D_T$, the mapping
$\varphi\rightarrow u(x, t)$ is a continuous nonnegative linear functional on $C(\pa_p D_T)$ and hence
there exists a unique Borel measure, $\omega^{(x,t)} = \omega^{(x,t)}_{D_T, L}$, on $\partial_{p} D_T$ so that 
\[
H_{\varphi}^{D_T}(x,t)=\int_{\partial_{p} D_T}\varphi(y,s)d\omega^{(x,t)}(y,s),\quad\text{ for every }\, \vf \in C(\pa_p D_T).
\] 
The measure $\omega^{(x,t)}$ is called the $L$-parabolic measure (associated with $D_T$) evaluated at
$(x, t)$. Because of Harnack's inequality, see \cite{KS} and \cite{BBLU2}, for nonnegative solutions of the heat
equation (\ref{HeatOp}), we have $\omega^{(x,s)} << \omega^{(y,t)}$ provided $x, y\in \Om$ and $0<s<t<T$. 

In this work we examine the relation between  the $L$-parabolic 
measure, $d\omega^{(x,t)}$, and the measure $d\sigma_X\times dt$ on $S_T$, where  $d\sigma_X=P_X(\Omega,\cdot)$ is the so-called horizontal perimeter measure  on $\pa\Omega$. The latter constitutes the appropriate replacement for the standard surface measure on $\pa\Omega$
and plays a central role in sub-Riemannian geometric measure theory. We also derive the solvability of the Dirichlet problem (\ref{Dirichlet-prob2}) for boundary data $\varphi\in L^p\left(S_T,d\sigma_X\times dt\right)$. 

The results in this paper should be considered as the parabolic counterpart of those in the paper \cite{CGN4}. As pointed out in that paper, a remarkable new aspect is  the interplay between the sub-Riemannian geometry associated with $X$ and the nature of the boundary of the domain $\Omega$. In this connection, those points of $\pa\Omega$ at which the vector fields  become tangent to $\pa\Omega$ play a special role. At such points, which are known as characteristic points,  solutions of $Lu = 0$ reveal a behavior that is quite different  from classical caloric functions. For a detailed discussion of this aspect in the time-independent case the reader should see \cite{CGN4}.

In order to describe this new behavior  we recall that given a bounded $C^1$ domain $\Omega\subset\Rn$, a point $x_0\in\pa\Omega$ is called characteristic for the system $X=\left\{X_1,\ldots,X_m\right\}$ if, indicating with $N(x_0)$ a normal vector to $\pa\Omega$ at $x_0$, one has
\[
\left\langle N(x_0),X_1(x_0) \right\rangle=\ldots=\left\langle N(x_0),X_m(x_0) \right\rangle=0.
\] 
The \emph{characteristic set} of $\Omega$, hereafter denoted by $\Sigma=\Sigma_{\Omega,X}$, is the collection of all characteristic points of $\pa\Omega$. It is a compact subset of $\pa\Omega$. 

In this paper  we prove the boundedness of the parabolic Poisson kernel for $L$ at the characteristic set, under
the hypothesis that the domain $\Omega\subset\Rn$ in (\ref{OmegaSet}) satisfies a condition similar to the classical Poincar\'e tangent
outer sphere condition, see \cite{P}. To understand the role of such assumption one should go back to the work of D. Jerison on the Dirichlet problem for the Kohn Laplacian on the Heisenberg group $\mathbb H^n$. 
In \cite{J},  Jerison constructed an example of a smooth (in fact, real analytic) domain for which the Dirichlet problem admits a Green's function which, in the neighborhood of an isolated characteristic point, is at most H\"older continuous up to the boundary. His domain is a paraboloid $\Om_M = \{(z,t)\in \mathbb H^n \mid t>-M|z|^2\}$, for an approriately chosen $M>0$. As a consequence, in Jerison's example the Poisson kernel for the Kohn Laplacian (that we introduce in the parabolic setting in Definition \ref{D:pk} below) fails to be bounded near the isolated characteristic point $(0,0)$. Besides showing, in particular, that Euclidean smoothness of the domain is not relevant in the subelliptic Dirichlet problem, a distinctive feature of Jerison's example is that it fails to satisfy a tangent outer ball condition at the characteristic point $(0,0)$. In fact,  being invariant with respect to the non-isotropic group dilations $(z,t)\to (\lambda z,\lambda^2 t)$, the domain $\Om_M$ should be interpreted as the analogue of a non-convex cone in Euclidean analysis.

The metric balls are not smooth (see \cite{CG1}), and therefore it would not be possible to have a notion of tangency
based on these sets. The right notion of tangency is,  as first shown in  \cite{CGN1}, \cite{LU1} and  \cite{CGN4},  based on smooth level sets of the time independent sub-Laplacian $-\sum^m_{j=1}X^*_jX_j$. These level sets are called $X-balls$. In \cite{CGN1}, \cite{LU1} and \cite{CGN4}, it was shown that, when the domain satisfies the tangent outer $X$-ball condition, then it is possible to control near the characteristic set the intrinsic horizontal gradient of the Green's function associated to a sub-Laplacian. 

In this paper we establish a similar result for the parabolic equation \eqref{HeatOp} above.
Beside the outer $X-$ball condition,  we  also assume that the domain $\Omega\subset\Rn$ in (\ref{OmegaSet}) is non-tangentially accessible with respect to the Carnot-Carath\'eodory metric associated with the system $X$ (a $NTA$ domain, henceforth), see Definition  
\ref{NTA-def} below. We also assume that $\Om$ be $C^{\infty}$. The $NTA$ property allows us to use some of the basic results developed
in \cite{CG1}, \cite{FrG} and \cite{M1}, whereas the smoothness assumption permits us to use, away from the
characteristic set, standard tools from calculus.

The $NTA$ domains, first introduced in the study of the Dirichlet problem for the standard Laplacian by Jerison and Kenig in \cite{JK},  appear to be suitable for  the study of boundary value problems in sub-Riemannian geometry. The notion of $NTA$ domain presents the advantage of being purely metrical, and therefore, in the context of the present paper, it is intrinsic to the underlying sub-Riemannian structure. The drawback is that producing rich classes of examples in this setting becomes much harder since the geometry is considerably more complicated than the Euclidean one. Fortunately, however, this has been done in the papers \cite{CG1}, \cite{MM1}, see also \cite{DGN2}.

Before summarizing the assumptions that we impose on the cylinder $D_T$, we prompt the interested reader to become acquainted with the notion of relative $X$-perimeter $P_X\left(E,A\right)$ of a Borel set $E\subset \Rn$ with respect to a given open set $A\subset \Rn$ introduced in \cite{CDG2}. This constitutes a generalization of De Giorgi's notion of intrinsic perimeter. As such, it represents the appropriate
``surface measure'' in sub-Riemannian geometric measure theory. With such notion of $X$-perimeter, we define the $X$-surface measure $\sigma_X$ on the boundary of the cross-section $\Om\subset \Rn$ of the cylinder $D_T = \Om \times (0,T)$. Let $x\in \pa \Om$ and $r>0$, then the $\sigma_X$ measure of the metric surface ball $\pa \Om \cap B_d(x,r)$ is defined in the following way:   
\[
\sigma_{X}\left(\pa\Omega\cap B_d(x,r)\right)\overset{def}{=} P_X\left(\Omega,B_d\left(x,r\right)\right).
\]

\begin{definition}
\label{ADP}
Given a system $X = \left\{X_1, ...,X_m\right\}$ of smooth vector fields satisfying (\ref{frc}), we
say that a connected bounded open set $\Omega\subset\Rn$ is \emph{admissible} with
respect to the system $X$, or simply $ADP_X$, if:
\begin{enumerate}
	\item[1)]  $\Omega$ is of class $C^{\infty}$;
	\item[2)]  $\Omega$ is non-tangentially accessible (NTA) with respect to the Carnot-Carath\'eodory metric
associated to the system $X$ (see Definition \ref{NTA-def} below);
  \item[3)]  $\Omega$ satisfies a uniform tangent outer X-ball condition (see Definition \ref{Def-X-ball} below);
  \item[4)]  The measure $\sigma_X$ is upper $1$-Ahlfors regular. This means that there exist
$A,R_0 > 0$ depending on $X$ and $\Omega$ such that for every $x\in\pa\Omega$ and $0<r<R_0$ one has
\[
\sigma_{X}\left(\pa\Omega\cap B_d(x,r)\right)\leq A \frac{\left|B_d(x, r)\right|}{r}.
\]
\end{enumerate}
When the condition $4)$ is replaced by the following \emph{balanced degeneracy assumption} on the standard surface measure $\sigma$ on $\pa \Om$:
\begin{itemize}
\item[$4'$)] There exist $B,R_0>0$ depending on $X$ and $\Omega$ such that for every $x\in\pa\Omega$ and $0<r<R_0$ one has 
\[
\left(\max_{y\in \pa\Omega\cap B_d(x,r)} W(y)\right)\ \sigma\left(\pa\Omega\cap B_d(x,r)\right)\leq B \frac{\left|B_d(x, r)\right|}{r},
\]
\end{itemize}
then we say that $\Om$ is a $\sigma-ADP_X$ domain. 
\end{definition}
For the definition of the function $W$ in $4')$, see \eqref{af} below. We emphasize that arbitrary bounded $C^\infty$ domains need not to satisfy either conditions $4)$ or $4')$. For this aspect see \cite{CGN1}, \cite{CGN3}, \cite{DGN2}, \cite{CG2}, \cite{CGN4}. Furthermore, in these works it is shown that the standard surface measure $\sigma$ on $\pa \Om$ can behave quite badly at characteristic points. The one-sided Ahlfors property, namely assumption $4)$, will be needed to establish the mutual absolute continuity of  the $L$-parabolic measure and the horizontal perimeter measure, see Theorem \ref{HolderK} below.  Condition $4')$ instead, will be important in the proof of Theorem \ref{HolderP} below. 

The constants appearing in $4)$ (or in $4')$), and in Definitions \ref{Def-X-ball} and \ref{NTA-def} below, will be referred to as the
$ADP_X$-parameters (or, the $\sigma-ADP_X$ parameters) of $\Omega$. We introduce next a central character in this play, the parabolic 
Poisson kernel of $D_T=\Omega\times(0,T)$ and $L$. In fact, we define two such functions, each one playing a different role. Assume that $\Om\subset \Rn$ be of class $C^\infty$ and let
$G(x,t;y,s) = G_{D_T}^L(x,t;y,s)$ indicate the Green function for the heat operator $L$ in (\ref{HeatOp}) and for a cylinder $D_T = \Om \times (0,T)$. 
By H\"ormander's theorem \cite{H} and the results in \cite{KN1} and \cite{De1}, see Theorem \ref{Derr}
below, for any fixed $(x,t)\in D$ the function $(y,s)\rightarrow G(x,t;y,s)$ is $C^{\infty}$ up to the boundary in a suitably
small neighborhood of any non-characteristic point. Let $\nu(y)$ indicate the outer unit
normal at $y\in\pa\Omega$. At every point $y\in\pa\Omega$ we denote by $N^X(y)$ the vector defined by
\[
N^{X}\left(y\right)=\left(\left\langle \nu\left(y\right),X_{1}\left(y\right)\right\rangle,\dots,\left\langle \nu\left(y\right),X_{m}\left(y\right)\right\rangle\right).
\]
Denoting by $\Sigma$ the characteristic set of $\Omega$, we remark that the vector $N^X(y)=0$ if and
only if $y\in\Sigma$. 
 We also consider the so-called angle function
\begin{equation}\label{af}
W\left(y\right)=\left|N^{X}\left(y\right)\right|=\sqrt{\sum^{m}_{j=1}\left\langle \nu\left(y\right),X_{j}\left(y\right)\right\rangle^{2}}.
\end{equation}
\noindent We note explicitly that it was proved in \cite{CDG2} that on $\pa\Omega$
\[
d\sigma_X=Wd\sigma.
\]
This formula implies, in particular, that it is always true that $d\sigma_X << d\sigma$, where $d\sigma$ denotes the standard surface measure on $\pa \Om$, but the opposite is  not true in general since $d\sigma$ can be quite singular on the characteristic set.
\begin{definition}\label{D:pk}
Given a $C^{\infty}$ bounded open set $\Om\subset\mathbb{R}^n$, and $D_T$ as in \eqref{OmegaSet} above, for every $((x,t),(y,s))\in D_T \times((\pa\Omega\setminus\Sigma)\times \left(0,T\right) )$, we define the \emph{parabolic Poisson kernel for $D_T$ and $L$} as follows 
\begin{equation*}
P(x,t;y,s)=
\left\{
\begin{array}{ll}
<XG\left(x,t;y,s\right),N^{X}\left(y\right)>, & 0<s<t,\\
0, & s\geq t 
\end{array} \right.
\end{equation*}
We also define 
\begin{equation*}
K(x,t;y,s)=\frac{P(x,t;y,s)}{W(y)}.
\end{equation*}
\end{definition}

We emphasize here that the reason for which in the definition of
$P(x,t;y,s)$ and $K(x,t;y,s)$ we restrict to points $(y,s)\in (\pa\Omega\setminus\Sigma)\times \left(0,T\right) $  is
that, as we have explained above,  the horizontal gradient $XG(x,t;y,s)$ may not be
defined at points of $\Sigma\times \left(0,T\right)$. Since the angle function
$W$ vanishes on $\Sigma \times \left(0,T\right)$, it should be clear that the
function $K(x,t;y,s)$ is more singular then $P(x,t;y,s)$ at the
characteristic points. However, such additional singularity is
balanced by the fact that the density $W$ of the measure $\sigma_X$
with respect to surface measure on $\pa \Om$ vanishes at the characteristic
points. As a consequence, $K(x,t;y,s)$ is the appropriate 
Poisson kernel with respect to the intrinsic measure $d\sigma_X\times dt$,
whereas $P(x,t;y,s)$ is more naturally attached to the ``wrong measure"
$d\sigma\times dt$.

Hereafter, for $(x,t)\in S_T$ it will be convenient to indicate with 
\[
C_r(x,t)=B_d(x,r)\times (t-r^2,t+r^2)
\]
the parabolic cylinder   centered at $(x,t)$ with radius $r>0$. We also denote by 
\[
\Delta_{r}(x,t)=C_{r}(x,t)\cap S_T
\]
the parabolic lateral surface cylinder centered at $(x,t)$ with radius $r>0$. The first main result in this paper is
contained in the following theorem. 
\begin{thm}
\label{HolderK}
Let $D_T=\Omega\times(0,T)$, where $\Omega$ is an $ADP_X$ domain and $T>0$. For every $p > 1$ and any fixed $x\in \Om$ there
exist positive constants $C,R_1,$ depending on $p,M,R_0, (x,T),$ and on the $ADP_X$ parameters, such
that for $(x_0,t_0)\in S_T$, with $0<r<R_1$ and $\Delta_r(x_0,t_0)\subset S_T$, one has 
\begin{eqnarray*}
\left(\frac{1}{\left|\Delta_r(x_0,t_0)\right|_{d\sigma_X\times dt}}\int_{\Delta_r(x_0,t_0)}K(x,T; y, t)^p d\sigma_X(y) dt\right)^{\frac{1}{p}} \leq \\ \frac{C}{\left|\Delta_r(x_0,t_0)\right|_{d\sigma_X\times dt}}\int_{\Delta_r(x_0,t_0)}K(x,T;y,t) d\sigma_X(y)dt,
\end{eqnarray*}
where $\left|\Delta_r(x_0,t_0)\right|_{d\sigma_X\times dt}=\int_{\Delta_r(x_0,t_0)}d\sigma_X(y)dt$. Moreover, the $L$-parabolic measure $\omega^{(x,T)}$ and the measure $d\sigma_X \times dt$  are
mutually absolutely continuous on $S_T$.
\end{thm}


A significant class of Carnot groups of step two in which one can construct examples of $ADP_X$ domains is that of groups of Heisenberg type. Such groups constitute a generalization of the Heisenberg group and they carry a natural complex structure. Let $\mathfrak{g}$ be a graded nilpotent Lie algebra of step two. This means that $\mathfrak{g}$  admits a splitting $\mathfrak{g} = V_1\oplus V_2$, where $[V_1,V_1] = V_2$, and $[V_1, V_2] = {0}$.  We endow $\mathfrak{g}$ with an inner product $\langle \cdot,\cdot\rangle$ with respect to which the decomposition $V_1\oplus V_2$ is orthogonal. Let $\mathbb{G}$ be the connected and simply connected graded nilpotent Lie group associated with $\mathfrak{g}$. Consider the map $J : V_2\rightarrow \text{End}(V_1)$ defined for every $\eta \in V_2$ by
\[
\langle J(\eta)\xi,\xi' \rangle=\langle[\xi,\xi'],\eta\rangle, \quad \xi,\xi' \in V_1,\eta \in V2.
\]
Then, $\mathbb{G}$ is said of H-type (Heisenberg type) if $J(\eta)$ is an orthogonal map on $V_1$ for every $\eta \in V_2$ such that $\|\eta\| = 1$.
When $\mathbb{G}$ is of H-type we thus have for $\xi,\xi' \in V_1,\eta\in V_2,$
\[
\langle J(\eta)\xi, J(\eta)\xi' \rangle= \|\eta\|^2 \langle \xi, \xi' \rangle .
\]
The $J$ map induces a complex structure since in every group of H-type one has for every $\eta,\eta' \in V_2$,
\[
J(\eta)J(\eta')+J(\eta')J(\eta)=-2\langle\eta,\eta'\rangle I,
\]
see \cite{K}. In particular,
\[
J(\eta)^2 =-\|\eta\|^2 I.
\]

We mention explicitly that, thanks to the results in \cite{K}, in every group of Heisenberg type with an orthogonal system $X$ of generators of $\mathfrak{g}= V_1\oplus V_2$, the fundamental solution of the sub-Laplacian associated with $X$ is given by
\[
\Gamma(x,y) = \frac{C(\mathbb{G})}{N(x,y)^{Q-2}}, 
\]
where $Q=\text{dim }V_1+2\text{ dim } V_2$ is the homogeneous dimension of $\mathbb{G}$,  and
\[
N(x,y)=\left(|x|^4+16|y|^2\right)^{1/4}.
\]
The following result provides a general class of domains satisfying the uniform $X$-ball condition, see  \cite{CGN1},  \cite{CGN2}, \cite{CGN3}, \cite{CGN4}, and also \cite{LU1}. We recall the following definition from \cite{CGN3}. Given a Carnot group $\mathbb{G}$, with Lie algebra $\mathfrak{g}$, a set $A \subset \mathbb{G}$ is called convex, if $\exp^{-1}(A)$ is a convex subset of $\mathfrak{g}$.

\begin{thm} 
\label{ADP-dom}
Let $\mathbb{G}$ be a Carnot group of Heisenberg type and denote by $X = \{X_1,\ldots,X_m\}$ a set of generators of its Lie algebra. Every $C^{\infty}$ convex bounded domain $\Omega\subset \mathbb{G}$ is a $ADP_X$ and also a $\sigma-ADP_X$ domain. In particular, the gauge balls in $\mathbb{G}$ are $ADP_X$ and also $\sigma-ADP_X$ domains.
\end{thm}

By combining Theorem \ref{HolderK} and Theorem \ref{ADP-dom} with the results in \cite{M1} and \cite{FrG}, we can solve the Dirichlet problem with boundary data in $L^{p}$ with respect to the measure $d\sigma_X \times dt$ in cylinders $D_T = \Om \times (0,T)\subset \mathbb{G}\times (0,T)$, where $\mathbb{G}$ is a group of Heisenberg-type and $\Om$ is a $C^{\infty}$ convex bounded domain. To state the relevant results
we need to introduce a definition. For any $(y,s)\in S_T$ a
nontangential region at $(y,s)$ is defined by
\[
\Gamma_{\alpha}(y,s)=\left\{(x,t)\in D_T: d_{p}((x,t),(y,s))<(1+\alpha)d_{p}((x,t),S_T)\right\},
\]
where
\[
d_{p}((x,t),(y,s))=\sqrt{d(x,y)^2+\left|t-s\right|}.
\]
Given a function $u\in C(D_T)$, the $\alpha-$nontangential maximal function of $u$ at $y$ is defined by
\[
N_{\alpha}(u)(y,s)=\sup_{(x,t)\in\Gamma_{\alpha}(y,s)}\left|u(x,t)\right|.
\] 
In the following result $L$ indicates the heat equation associated with any sub-Laplacian on $\mathbb G$.
\begin{thm}
\label{NonTanK}
Let $\mathbb{G}$ be a Carnot group of H-type and $\Om\subset\mathbb{G}$ be a $C^{\infty}$ convex bounded domain. Set $D_T=\Om \times(0,T)$, with $0<T<\infty$. Then, for every $p>1$, if $f\in L^{p}\left(S_T,d\sigma_X\times dt\right)$, the generalized solution $H^{D_T}_f$ to the problem 
\[
\left\{
\begin{array}{ll}
Lu=0 & \mbox{in $D_T$},\\
u= f & \mbox{on $S_T$},\\
u=0 & \mbox{on $\Om\times\left\{0\right\}$} 
\end{array} \right.
\]
exists, and it is given by
\[
H^{D_T}_f(x,t)=\int_{S_T} K(x,t,y,s)f(y,s)d\sigma_X(y)ds.
\]
Moreover,  there exists a constant $C>0$, depending on $\mathbb G, \Om$ and $p$, such that 
\[
\left\|N_{\alpha}\left(H^{D_T}_f\right)\right\|_{L^p\left(S_T,d\sigma_Xdt\right)}\leq C\left\|f\right\|_{L^p\left(S_T,d\sigma_Xdt\right)}.
\]
Furthermore, $H^{D_T}_f$ converges nontangentially $d\sigma_X\times dt-$a.e. to $f$ on $S_T$.
\end{thm}

Theorems \ref{HolderK} and \ref{NonTanK} provide solvability of the Dirichlet problem with respect to the ``natural"  surface measure $d\sigma_X\times dt$ on $S_T$. On the other hand, even if, as we have mentioned, the ordinary surface measure $d\sigma\times dt$ is the ``wrong one'' in the Dirichlet problem, it would still be highly desirable to know if there exist situations in which
(\ref{Dirichlet-prob2}) can be solved for boundary data in some $L^p$ with respect to $d\sigma\times dt$. This is where the notion of $\sigma-ADP_X$ domain becomes relevant.


\begin{thm}
\label{HolderP}
Let $D_{T_1}=\Omega\times(0,T_1)$, where $\Omega$ is a $\sigma-ADP_X$ domain and $T_1>0$. For every $p > 1$ and any fixed $(x,T)\in D_{T_1}$ there
exist positive constants $C,R_1,$ depending on $p,M,R_0, (x,T),$ and on the $\sigma-ADP_X$ parameters, such
that for that for $(x_0,t_0)\in S_T$, with $0<r<R_1$ and $\Delta_r(x_0,t_0)\subset S_T$, one has 
\begin{eqnarray*}
\left(\frac{1}{\left|\Delta_r(x_0,t_0)\right|_{d\sigma\times dt}}\int_{\Delta_r(x_0,t_0)}P(x,T; y, t)^p d\sigma_X(y) dt\right)^{\frac{1}{p}} \leq \\ \frac{C}{\left|\Delta_r(x_0,t_0)\right|_{d\sigma\times dt}}\int_{\Delta_r(x_0,t_0)}P(x,T;y,t) d\sigma_X(y)dt,
\end{eqnarray*}
where $\left|\Delta_r(x_0,t_0)\right|_{d\sigma\times dt}=\int_{\Delta_r(x_0,t_0)}d\sigma dt$. Moreover, the $L$-parabolic measure $\omega^{(x,T)}$ and the measure $d\sigma \times dt$ on $S_T$ are
mutually absolutely continuous.
\end{thm}

Combining Theorems \ref{HolderP} and \ref{ADP-dom}, we obtain the following result.

\begin{thm}
\label{NonTanP}
Let $\mathbb{G}$ be a Carnot group of Heisenberg-type and $\Om\subset\mathbb{G}$ be a bounded $C^\infty$ convex set. Set $D_T=\Om \times(0,T)$, with $0<T<\infty$. For every $p>1$, and $f\in L^{p}\left(S_T,d\sigma\times dt\right)$, then the generalized solution $H^{D_T}_f$ to the problem 
\[
\left\{
\begin{array}{ll}
Lu=0 & \mbox{in $D_T$},\\
u= f & \mbox{on $S_T$},\\
u=0 & \mbox{on $\Om\times\left\{0\right\}$} 
\end{array} \right.
\]
exists, and it is given by
\[
H^{D_T}_f(x,t)=\int_{S_T} P(x,t,y,s)f(y,s)d\sigma(y)ds.
\]
Moreover,  there exists a constant $C>0$ depending on $\mathbb G, \Om$ and $p$, such that 
\[
\left\|N_{\alpha}\left(H^{D_T}_f\right)\right\|_{L^p\left(S_T,d\sigma dt\right)}\leq C\left\|f\right\|_{L^p\left(S_T,d\sigma dt\right)}.
\]
Furthermore, $H^{D_T}_f$ converges nontangentially $d\sigma\times dt-$a.e. to $f$ on $S_T$.
\end{thm}

In closing, we briefly describe the organization of the paper. In Section 2 we collect some
known results on Carnot-Carath\'eodory metrics that are needed in the paper. In Section 3 we
discuss some known results on the subelliptic Dirichlet problem which constitute the potential
theoretic backbone of the paper. 

In Section 4, we use the interior Schauder estimates in \cite{DG1} to prove that if the cross-section $\Om$ of the cylinder $D_T = \Om \times (0,T)$ satisfies
a uniform outer tangent $X$-ball condition, then the horizontal gradient of the Green function
$G$ is bounded up to the lateral boundary $S_T$, hence, in particular, near $\Sigma_T$, see Theorem \ref{HorBound-G}. The proof
of such result rests in an essential way on the linear growth estimate provided by an ad-hoc barrier, inspired by the one used in \cite{CGN1}, \cite{LU1} and \cite{CGN4}. In the final part of the section we show that, by imposing the uniform
outer $X$-ball condition only in a neighborhood of the characteristic set $\Sigma$, we are still able to
obtain the boundedness of the horizontal gradient of $G$ up to the characteristic set, although we
now lose the uniformity in the estimates, see Corollary \ref{HeatBall-N-Sigma}.

In Section 5, we establish a Poisson type representation formula for domains whose base satisfy
the uniform outer $X$-ball condition in a neighborhood of the characteristic set. This result
generalizes  the Poisson type formula in \cite {FSa} to the non-Euclidean setting of this paper. If the Green function of a smooth domain had bounded horizontal gradient up to the
characteristic set, then such Poisson formula would follow in an elementary way from integration
by parts. As we previously stressed, however, things are not so simple and the boundedness of
$XG$ fails in general near the characteristic set. However, when $\Omega\subset\Rn$ satisfies the uniform
outer $X$-ball condition in a neighborhood of the characteristic set, then combining Theorem \ref{HorBound-G}
with the estimate
\[
K(x,t; y,s)\leq \left|XG(x,t;y,s)\right| , x \in \Omega, y\in \pa\Omega ,\quad \text{and}\quad 0\leq s\leq t, 
\]
see (\ref{Bound-P-K}), we prove the boundedness of the Poisson kernel $(y,s)\rightarrow K(x,t;y,s)$ on $S_T$. The main result in
section 5 is Theorem \ref{Main-P-K}. Solvability of (\ref{Dirichlet-prob2}) with data in Lebesgue classes requires, however, a much deeper
analysis.

In the opening of Section 6 we recall the definition of $NTA$-domain along
with those results from \cite{M1}, \cite{FGGMN} and \cite{FrG} which constitute the foundations of the present study. Using
these results we establish Theorem 8.9. The remaining part of the section is devoted to proving
Theorems \ref{HolderK}, \ref{NonTanK}, \ref{HolderP} and \ref{NonTanP}. 
\section{Preliminaries}
\label{sec:2}
In $\R^n$, with $n\geq 3$, we consider a system  $X =
\{X_1,\ldots,X_m\}$ of $C^\infty$ vector fields satisfying
H\"ormander's finite rank condition \eqref{frc}. A piecewise $C^1$
curve $\gamma:[0,\ell]\to \R^n$ is called subunitary in \cite{FPh}
if whenever $\gamma'(t)$ exists one has for every $\xi\in\Rn$
\[
<\gamma'(t),\xi>^2\ \leq\ \sum_{j=1}^m <X_j(\gamma(t)),\xi>^2 .
\]
We note explicitly that the above inequality forces $\gamma '(t)$ to
belong to the span of $\{X_1(\gamma (t)),\ldots,$ $ X_m(\gamma (t))\}$.
The subunit length of $\gamma$ is by definition $l_s(\gamma)=\ell$.
If we fix an open set $\Om\subset \Rn$, then given $x, y\in \Om$, denote by $\mathcal S_\Om(x,y)$ the collection
of all sub-unitary $\gamma:[0,\ell]\to \Om$ which join $x$ to $y$. The fundamental
accessibility theorem of Chow and Rashevsky, \cite{Ch}, \cite{Ra},
states that, if $\Om$ is connected, then for every
$x,y\in \Om$ there exists $\gamma \in \mathcal S_\Om(x,y)$. As a
consequence, if we define
\[
d_{\Om}(x,y) = \text{inf}\ \{l_s(\gamma)\mid \gamma \in \mathcal S_\Om(x,y)\},
\]
we obtain a distance on $\Om$, called the Carnot-Carath\'eodory distance, associated with the system $X$. When $\Om = \Rn$, we write $d(x,y)$ instead of $d_{\Rn}(x,y)$. It is clear that $d(x,y) \leq d_\Om(x,y)$, $x, y\in \Om$, for every connected open set $\Om \subset \Rn$. In \cite{NSW} it was proved that, given $\Om \subset \subset \Rn$, there exist $C, \ve >0$ such that
\begin{equation}\label{CCeucl}
C |x - y| \leq d_\Om(x,y) \leq C^{-1} |x - y|^\ve , \quad\quad\quad x, y \in \Om.
\end{equation}
This gives $d(x,y)\ \leq C^{-1} |x - y|^\ve$, $x, y\in \Om$, and therefore
\[
i: (\Rn, |\cdot|)\to (\Rn, d) \quad\quad\quad \text{is\,\ continuous}.
\]
It is easy to see that also the continuity of the opposite inclusion holds \cite{GN1}, hence the metric and the Euclidean topologies are equivalent.

For $x\in \Rn$ and $r>0$, we let $B_d(x,r) = \{y\in \Rn\mid d(x,y) < r \}$. The basic properties of these balls were established by Nagel, Stein and Wainger in their seminal paper \cite{NSW}. Denote by $Y_1,...,Y_l$ the collection of the $X_j's$ and of those commutators which are needed to generate $\R^n$. A formal ``degree'' is assigned to each $Y_i$, namely the corresponding order of the commutator. If $I=(i_1,\ldots,i_n),1\leq i_j\leq l$ is a $n$-tuple of integers, following \cite{NSW} we let $d(I)=\sum^{n}_{j=1}\operatorname{deg}(Y_{i_j})$, and $a_I(x)=\det(Y_{i_1},\ldots,Y_{i_n}).$ The Nagel-Stein-Waigner polynomial is defined by 
\[
\Lambda(x,r) = \sum_I |a_I(x)| r^{d_I}
\]
For a given bounded open set $U\subset\Rn$, we let
\begin{equation}
\label{Q-Qx}
Q=\sup\left\{d(I)|\left|a_I(x)\right|\neq0,x\in U\right\},\quad Q(x)=\inf\left\{d(I)|\left|a_I(x)\right|\neq0\right\},
\end{equation}
and notice that $n\leq Q(x)\leq Q$. The numbers $Q$ and $Q(x)$ are respectively called the \emph{local homogeneous dimension of $U$} and the homogeneous dimension at $x$ with respect to the system $X$.
\begin{thm}
For every  bounded open set $U\subset \Rn$, there exist constants $C, R_0>0$ such that, for any $x\in U$, and $0 < r \leq R_0$,
\begin{equation}\label{pol}
C\le \frac{\left|B_d(x,r)\right|}{\Lambda(x,r)} \le C^{-1}.
\end{equation}
As a consequence, one has with $C_1>0$, 
\begin{equation}\label{doubling}
|B_d(x,2r)| \leq C_1 |B_d(x,r)| \qquad\text{for every}\quad x\in U\quad\text{and}\quad 0< r \leq R_0.
\end{equation}
\end{thm}

Let $\Gamma\left(x,y\right)=\Gamma\left(y,x\right)$ be the positive fundamental solution of the sub-Laplacian $-\sum_{i=1}^{m}X_{i}^{\ast }X_{i}$
and consider its level sets
\[
\Omega(x,r)=\left\{y\in\mathbb{R}^n:\Gamma(x,y)>\frac{1}{r}\right\}.
\]
Let us recall the notion of $X$-ball, introduced in \cite{CGN1}.
\begin{definition}
For every $x\in \mathbb{R}^{n}$, and $r>0$, the set 
\begin{equation}
\label{X-ball}
B\left(x,r\right)=\left\{y \in \mathbb{R}^{n}: \Gamma\left(x,y\right)>\frac{1}{E\left(x,r\right)}\right\},
\end{equation}
where $E(x,r)\overset{def}{=} \Lambda(x,r)/r^2$, will be called the $X$-ball, centered at $x$ with radius $r$.
\end{definition}
The $X$-balls are equivalent to the Carnot-Carath\'eodory balls: for every $U\subset\mathbb{R}^{n}$, there exists $a>1$, depending on $U$ and $X$, such that 
\begin{equation}
\label{X-balls}
B_{d}\left(x,a^{-1}r\right)\subseteq B\left(x,r\right)\subseteq B_{d}\left(x,ar\right),
\end{equation} 
for $x\in U$, $0<d\left(x,y\right)\leq R_{0}$, for some local parameter $R_{0}$, which depends only on $n$ and $X$. 

\subsection{Strong maximum principle}

Consider the cylinder $D_T=\Omega\times (0,T)$, where $\Omega\subset\Rn$ and $T>0$. We define the class $\Gamma^2(D_T)$ to be the set of all continuous functions $u$ on $D_T$ such that $\partial_t u$ as well as $X_iu$ and $X_iX_ju$ are continuous in $D_T$ for all $i,j=1,\ldots,m$.

\begin{thm}\label{strongmax}  
Let $\Omega \subset \R^{n}$ be a
connected, bounded open set, and  $T>0$. 
Let $u\in \Gamma ^{2}(D _{T})$ and $u\leq 0$ in $D_T$.
\begin{enumerate}
 \item Suppose that $Lu\geq 0$ in $D_{T}$, then, if $u(x_{0},t_{0})=0$ for some $(x_{0},t_{0})\in
D_{T}$, we have $u(x,t)\equiv 0$ whenever $(x,t)\in D _{T}\cap
\{t:t\leq t_{0}\}$.
\item Suppose that $L^*u\geq 0$ in $D_{T}$, then, if $u(x_{0},t_{0})=0$ for some $(x_{0},t_{0})\in
D_{T}$, we have $u(x,t)\equiv 0$ whenever $(x,t)\in D _{T}\cap
\{t:t\geq t_{0}\}$.

\end{enumerate}
\end{thm}
\begin{proof}
It follows from
Theorem 3.2 in \cite{Bony}.
\end{proof}

\subsection{Gaussian bounds}

We close this section with the following basic estimate due to Kusuoka and Stroock \cite{KS}, see also \cite{JSC}. In order to use this result we need to assume  that outside a compact set $K\subset\R^n$, with  $\Omega\subset K$,  the system $X$ equals the standard basis for $\Rn$. Since our results are local in nature the imposition of such  hypothesis on the system $X$ has no consequences for our theory.
\begin{thm}
\label{GaussB}
The fundamental solution $p(x,t;\xi ,\tau )=p(x;\xi ,t-\tau )$ with
singularity at $(\xi ,\tau )$ satisfies the following size estimates: there
exists $M=M\left( X\right) >0$ and for every $k,s\in \mathbb{N}\cup \left\{
0\right\} $, there exists a constant $C=C\left( X,k,s\right)>0 $, such that 
\begin{equation*}
\left\vert \frac{\partial {^{k}}}{\partial {t^{k}}}
X_{j_{1}}X_{j_{2}}...X_{j_{s}}p(x,t;\xi ,\tau )\right\vert \leq \frac{C}{
\left( t-\tau \right) ^{s+2k}}\frac{1}{\left\vert B_{d}\left( x,\sqrt{t-\tau }
\right) \right\vert }\exp \left( -\frac{Md\left( x,\xi \right) ^{2}}{t-\tau }
\right),
\end{equation*}
\begin{equation*}
p(x,t;\xi ,\tau )\geq \frac{C^{-1}}{\left\vert B_{d}\left( x,
\sqrt{t-\tau }\right) \right\vert }\exp \left( -\frac{M^{-1}d\left( x,\xi
\right) ^{2}}{t-\tau }\right),
\end{equation*}
for every $x,\xi \in \mathbb{R}^{n}$ and any $-\infty <\tau <t<\infty$.
\end{thm}  

\section{Dirichlet Problem}

In the following we let $D\subset \R^{n+1}$ be any bounded open set, for which we assign $f\in C(\partial_pD)$,  and we
study the Dirichlet problem 
\begin{equation}\label{dirp}
\begin{cases}
Lu=0,\ \ \ \ \text{in}\ D,
\\
u=f \ \ \ \ \text{on}\  \partial_p D.
\end{cases}
\end{equation}
If $u:D\to\R$ is a smooth function satisfying 
$L u=0$ in $D$, then we say that $u$ is $L$-parabolic in $D $. We
denote by $\mathcal P_L(D)$ the linear space of functions which are $L$-parabolic in $D$. 

We 
say that $D$ is $L$-regular if for any $f\in C(\partial_p D)$ there exists a unique function $H_f^D\in \mathcal P_L(D)$ such that $\lim_{(x,t)\to
(x_0,t_0)}H_f^D(x,t)=f(x_0,t_0)$ for every $(x_0,t_0)\in\partial_p D$.
Furthermore, if $D$ is $L$-regular, then in view of Theorem \ref{strongmax} (one actually only needs the weaker form of it) for every fixed $(x,t)\in D$ the map $f\mapsto H_f^D(x,t)$ defines a positive linear functional on $C(\partial_p D)$.
By the Riesz representation theorem there exists a unique Borel
measure $\omega^{(x,t)} =\omega_D^{(x,t)}$, supported in $\partial_p D$, such that for every $f\in C(\partial_p D)$ one has
\begin{equation}  \label{parabolic.measure.L.regular}
H_f^D(x,t)=\int_{\partial_p D}f(y,s)d\omega^{(x,t)}(y,s).
\end{equation}
We will refer to $\omega^{(x,t)}$ as the $L$-\emph{parabolic measure} relative to $D$
and $(x,t)$. 

A lower semi-continuous
function $u:D\to\,]-\infty,\infty]$ is said to be $H$-superparabolic
in $D$ if $u<\infty$ in a dense subset of $D$ and if 
\begin{equation*}
u(x,t)\geq\int_{\partial V}u(y,s)d\omega_V^{(x,t)}(y,s),
\end{equation*}
for every open $L$-regular set $V\subset\overline{V}\subset D$ and for every 
$(x,t)\in V$. We denote by $\overline{S}(D)$ the set of $L$-superparabolic functions in $D$, and by $\overline{S}^+(D)$ the set of the
functions in $\overline{S}(D)$ which are nonnegative. A function $v:D\rightarrow [-\infty,\infty[$ is said to be $L$-subparabolic in $D$
if $-v\in\overline{S}(D)$ and we write $\underline{S}(D):=-\overline{S}(D)$.
As the collection of $L$-regular sets is a basis for the Euclidean topology
according to Corollary 5.2 in \cite{Bony}, it follows that $\overline{S}(D)\cap\underline{S}(D)= \mathcal P_L(D)$. Finally, we recall that $H_f^D$ can be realized as
the generalized solution in the sense of Perron-Wiener-Brelot-Bauer to the
problem \eqref{dirp}. In particular, 
\begin{equation}  \label{parabolic.measureumu}
\inf\overline{\mathcal{U}}_f^D=\sup\underline{\mathcal{U}}_f^D=H_f^D,
\end{equation}
where we have indicated with $\overline{\mathcal{U}}_f^D$ the collection of all $u\in\overline{S}(D)$ such that $\inf_{D}u>-\infty$, 
and
\[
\liminf_{(x,t)\to(x_0,t_0)}u(x,t)\geq f(x_0,t_0),
\,\forall\,(x_0,t_0)\in\partial_p D,
\]
 and with $\underline{\mathcal{U}}_f^D$ the collection of all $u\in\underline{S}(D)$ for which $\sup_{D}u<\infty$, and 
\[
\limsup_{(x,t)\to (x_0,t_0)}u(x,t)\leq f(x_0,t_0), \quad \text{for every $(x_0,t_0)\in\partial_p D$}.
\]

\begin{lemma}\label{gendir} 
Let $D\subset\R^{n+1}$ be a bounded open set, let $f\in
C(\partial_pD)$, and let $u$ be the generalized
Perron-Wiener-Brelot-Bauer solution to the problem  \eqref{dirp}, i.e., $u=H_f^D$
where $H_f^D$ be defined as in \eqref{parabolic.measureumu}. Then, $u\in
\Gamma^2(D)$.
\end{lemma}

\begin{proof}
This follows from Theorem 1.1 in \cite{U}. 
\end{proof}
Here, it is important to recall that, thanks to the results in \cite{Bony}, the following result of Brelot type holds.
\begin{thm}
\label{Brelot}
A function $f$ is resolutive if and only if $f\in L^1\left(\pa_p D,d\omega^{(x,t)}\right)$, for one (and therefore for all) $\left(x,t\right)\in D$.
\end{thm}
In the following we are concerned with the issue of regular boundary points
and we note, concerning the solvability of the Dirichlet problem for the
operator $L$, that in \cite{U} Uguzzoni developes what he refers to as a ``cone
criterion'' for non-divergence equations modeled on H{\"o}rmander vector
fields. This is a generalization of the well-known positive density condition of classical potential theory. In the following we describe his result in the setting of domains of
the form $D_T=\Omega\times (0,T)$ where $\Omega\subset \R^{n}$ is
assumed to be a bounded domain. In \cite{U}
a bounded open set $\Omega$ is said to have \emph{outer positive $d$-density at $x_0\in\partial\Omega$} if there exist $r_0$, $\vt>0$ such that 
\begin{equation}  \label{cone}
|B_d(x_0,r)\setminus\overline\Omega|\geq\vt |B_d(x_0,r)|,\ \ \text{for all}\  
r\in (0,r_0).
\end{equation}
Furthermore, if $r_0$ and $\vt$ can be chosen independently of $x_0$ then one says that $\Omega$ satisfies the outer positive $d$-density condition. The
following lemma is a special case of Theorem 4.1 in \cite{U}.

\begin{lemma}\label{Dirichlet0} 
Assume that $\Omega$ satisfies the outer positive $d$-density condition. Given $f\in C(\partial_p D_T)$, there exists a unique solution $u\in\Gamma^{2}(D_T)\cap C(D_T\cup\partial_p D_T)$ to the problem  \eqref{dirp}. In particular, $D_T$ is $L$-regular for the Dirichlet problem \eqref{dirp}.
\end{lemma}
For a cylinder $D_T=\Omega\times \left(0,T\right)$, where $\Omega\subset \mathbb{R}^{n}$ is open and bounded, the Green function, $G$, is given by $G\left(x,t;y,s\right)=p\left(x,t;y,s\right)-h_{\left(y,s\right)}\left(x,t\right)$ where $h_{\left(y,s\right)}$ solves
\begin{equation*}
\left\{
\begin{array}{ll}
Lh_{\left(y,s\right)}=0 & \mbox{in $D_T$},\\
h_{\left(y,s\right)}\left(x,t\right)= p\left(x,t;y,s\right) & \mbox{on $\partial_{p} D_T$}.
\end{array} \right.
\end{equation*}
The existence of the Green function follows from the results of \cite{Bony} and \cite{BBLU2}. In the following result we use the short notation $z = (x,t), \zeta = (y,s)$.
\begin{thm}
\label{Green-th}
Let $\Omega\subset\R^n$ be a regular domain for the sub-Laplacian $-\sum_{j=1}^m X_j^\star X_j$. Then, there exists a Green function $G = G^{\Omega}$ for $L$ and the cylinder $D_T = \Omega\times (0,T)$ with the properties listed below:
\begin{enumerate}
	\item G is a continuous function defined on the set $\left\{(z;\zeta)\in (\overline{\Omega}\times [0,T))\times(\Omega\times (0,T)) :
z\neq\zeta\right\}$. Moreover, for every  fixed $\zeta\in D_T, G(\cdot;\zeta)\in C^{\infty}(D_T\setminus\left\{\zeta\right\})$,
and we have
\[
L(G(\cdot; \zeta)) = 0 \quad \mbox{in $(D_T)\setminus\left\{\zeta\right\}$}, \quad G(\cdot;\zeta) = 0 \quad\mbox{in $\pa\Omega\times (0,T)$}
\]
  \item We have $0\leq G\leq p$. Moreover, $G(x,t;y,s) = 0$ if $t<s$.
  \item For every $\varphi\in C(\overline{\Omega})$ such that $\varphi=0$ in $\pa\Omega$ and for every fixed $s\in (0,T)$ the function
  \[
  u(x,t)=\int_{\Omega}G(x,t;y,s)\varphi(y)dy, \quad x\in\overline{\Omega}, t>s
  \]
  belongs to the class $C^{\infty}(\Omega\times(s,T))\cap C(\left[s,T\right)\times\overline{\Omega})$ and solves
  \begin{equation*}
\left\{
\begin{array}{ll}
Lu=0 & \mbox{in $ \Omega\times\left(s,T\right)$}, \\
u=0 & \mbox{in $ \pa\Omega\times\left[s,T\right)$}, \\
u(\cdot,s)=\varphi&  \mbox{in $\overline{\Omega}$}.
\end{array} \right.
\end{equation*}

\end{enumerate}
\end{thm}

We close this section with an important consequence of the results of Derridj \cite{De1} about smoothness in the Dirichlet problem at non-characteristic points. 



Given a closed set $S\subset \R^{n+1}$, we will denote by $C^{\infty}(S)$ the collection of the restrictions to $S$ of functions $\varphi\in C^{\infty}(U)$, where $U\subset\R^{n+1}$ is open with $S\subset U$.

\begin{thm}
\label{Derr}
Let $\Omega\subset\R^{n}$ be a $C^{\infty}$ domain which is regular for the sub-Laplacian $-\sum_{j=1}^m X_j^\star X_j$, and set $D_T=\Om \times \left(0,T\right)$. Consider the caloric function $H^{D_T}_{\vf}$, with $\vf\in C^{\infty}(\partial_pD_T)$. If $x_{0}\in\pa\Omega$ is a non-characteristic point, then there exists an open neighborhood $V$ of $x_{0}$ such that $H^{D_T}_{\vf}\in C^{\infty}((V\cap \overline \Om)\times (0,T))$.
\end{thm}

\section{Boundary Estimates for the Green Function}\label{S:be}

In this section we establish some estimates for the Green function as well as for its horizontal gradient near the bounday assuming some conditions on the geometry of the domain $\Omega$.

\begin{definition}
\label{Def-X-ball}
A domain $\Omega\subset\R^{n}$ is said to possess an outer $X$\emph{-ball} tangent at $x_{0}\in\pa\Omega$ if
for some $r > 0$ there exists a X-ball $B(x_{1},r)$ such that:
\begin{equation}
\label{X-ball-Tan}
x_{0}\in \pa B(x_{1},r), \quad B(x_{1},r)\cap \Omega=\varnothing.
\end{equation}
We say that $\Omega$ possesses the \emph{uniform outer $X$-ball} if one can find $R_{0}>0$ such that for every $x_{0}\in\pa\Omega$, and any $0<r<R_{0}$, there exists a X-ball $B(x_{1},r)$ for which (\ref{X-ball-Tan}) holds.
\end{definition}

Notice that, when
$X=\left\{\frac{\pa}{\pa x_{1}},\ldots,\frac{\pa}{\pa x_{n}}\right\}$, then the distance $d(x, y)$ is just the ordinary Euclidean distance $\left|x-y\right|$. In
such case, Definition \ref{Def-X-ball} coincides with the notion introduced by Poincar\'e in his classical paper
\cite{P}. In this setting a \emph{X}-ball is just a standard Euclidean ball, and, hence, a domain possesses  the uniform outer \emph{X}-ball condition if and only if it is $C^{1,1}$ smooth. In the sub-Riemannian setting, the \emph{gauge balls} in Carnot groups of Heisenberg-type satisfy the uniform outer $X$-ball condition. In fact, as it has already been mentioned in the introduction, every bounded convex subset of such Carnot groups satisfy the  uniform outer $X$-ball condition, see Theorem \ref{ADP-dom} and the original sources \cite{CGN2} and \cite{LU1}. Moreover, it should be clear from (\ref{X-balls}) and Theorem \ref{Dirichlet0} that if $\Omega\subset\Rn$ satisfies the uniform  outer
\emph{X}-ball condition, then $D=\Omega\times (0,T)$, $T>0$,  is regular for the Dirichlet problem (\ref{dirp}).

From Theorem \ref{GaussB},  and the maximum principle, one easily sees that there exists $C>0$ such that for every $(x,t),(y,s)\in D_{T}$
\begin{equation}
\label{Gauss-Green}
0\leq G(x,t;y,s)\leq \frac{C}{\left\vert B\left( x,\sqrt{t-s }
\right) \right\vert }\exp \left( -\frac{Md\left( x,y \right) ^{2}}{t-s }
\right).
\end{equation}

\begin{lemma}
\label{linear-growth}
Let $D_T=\Omega\times (0,T)$, where $\Omega\subset \mathbb{R}^n$ is a connected, bounded open set and $T>0$. Suppose that for some $r > 0$, $\Omega$ has an outer X-ball $B(x_1, r)$ tangent at $x_0 \in \partial \Omega$. Let 
\[
C_r=B(x_1,2r)\times (s_0,T),
\] 
for some $0\leq s_0< T$. There exists $C > 0$, depending only on $\Omega$ and on $X$, such that if $\varphi \in C(\partial_p D_T)$, $\varphi\equiv 0$ on $\left(B(x_1,2r)\cap\partial\Omega\right)\times(s_0,T)$ and $H^{D_T}_{\varphi}\equiv 0$ on $\left(B(x_1,2r)\cap\Omega\right)\times\{s_0\}$, then for every $(x,t)\in \Omega\times (s_0,T)$, we have
\[
|H^{D_T}_{\varphi}(x,t)| \leq C \frac{d(x,x_0)}{r} \max_{\partial_p D_T} |\varphi|.
\]
\end{lemma}
\begin{proof}
We can assume without loss of generality that $\max_{\partial_p D_T} |\varphi|=1$. Now, we consider  the following barrier
\begin{equation}
\label{subelliptic-barrier}
v(x,t)\overset{def}{=}\frac{E(x_{1},r)^{-1}-\Gamma(x_{1},x)}{E(x_{1},r)^{-1}-E(x_{1},2r)^{-1}}, 
\end{equation} 
with domain $\mathbb{R}^{n}\times\left(0,\infty\right)$. Clearly,  $Lv=0$ in $\left(\mathbb{R}^{n}\setminus\left\{x_{1}\right\}\right)\times\left(0,\infty\right)$ and $v\geq 0$ in $\left(\mathbb{R}^{n}\setminus B\left(x_{1},r\right)\right)\times\left(0,\infty\right)$. Furthermore, $v\equiv 1 $ on $\left(\partial B(x,2r)\cap \Omega\right)\times (0,T)$ and $v\geq 1$ in $\left(\mathbb{R}^{n}\setminus B\left(x_{1},2r\right)\right)\times\left(0,T\right)$. Therefore, the maximum principle implies
\begin{equation*}
|H^{D_T}_{\varphi}\left(x,t\right)|\leq v\left(x,t\right), \qquad \mbox{for every}\qquad \left(x,t\right)\in C_r\cap D.
\label{BoundH-f}
\end{equation*}
In order to finish the proof, we recall  from Theorem 6.3 in \cite{CGN1} that there exits $C=C(X,\Omega)>0$  such that
\begin{equation*}
v\left(x,t\right)\leq C\frac{d\left(x_{0},x\right)}{r}\qquad \mbox{for every} \qquad x\in\Omega.
\label{Boundf-d}
\end{equation*}
\end{proof}
\begin{thm}
Suppose that $\Omega$ satisfy the uniform outer X-ball condition with constant $R_{0}$. Then, there exist $C,M_1>0$, depending on $X$ and $\Omega$, such that the Green function for $L$ and $D_T = \Om \times (0,T)$ satisfies the estimate
\begin{align}
G\left(x,t;y,s\right)\leq \frac{Cd\left(y,\partial \Omega\right)}{\left(d\left(x,y\right)\wedge R_{0}\right)\left|B_{d}\left(x,\sqrt{t-s}\right)\right|}\exp\left(-\frac{M_1d\left(x,y\right)^{2}}{t-s}\right)\nonumber
\end{align}
with  $x\neq y\in \Omega$ and $0\leq s<t<T$.
\label{Bound-Est-G}
\end{thm}
\begin{proof}
Consider $a > 1$ as in (\ref{X-balls}). Because of (\ref{Gauss-Green}) we may assume that 
\[
ad(y,\partial\Omega)<\frac{d(x,y)}{a(a+3)}\quad \text{and} \quad d\left(y,\partial\Omega\right)<R_{0}.
\] 
Choose 
\begin{equation}
\label{Choose-r}
r=\min \left(\frac{d\left(x,y\right)}{2a\left(a+3\right)},\frac{aR_{0}}{2}\right).
\end{equation}
Let $x_{0}\in \partial \Omega$ such that $d\left(x_{0},y\right)=d\left(y,\partial\Omega\right)$. Let $B\left(x_{1},r/a\right)$ be the outer $X$-ball tangent to the boundary of $\Omega$ at $x_{0}$. Since
\[
x_{0}\in \overline{B(x_{1},r/a)}\subset \overline{B_{d}(x_{1},r)},
\]
we have
\begin{align}
d\left(y,x_{1}\right)\leq d\left(y,x_{0}\right)+d\left(x_{0},x_{1}\right) <\frac{2r}{a} + r<\frac{a+3}{a}r.
\end{align}
Therefore, 
\begin{equation}
y\in B_{d}\left(x_{1},a^{-1}(a+3)r\right)\subseteq B\left(x_{1},(a+3)r\right) .
\end{equation}
On the other hand, the triangle inequality gives
\begin{align*}
d\left(x,x_{1}\right)\geq d\left(x,y\right)-d\left(y,x_{1}\right)&\geq d\left(x,y\right)-\frac{a+3}{a}r\\
& > d\left(x,y\right)\left(1-\frac{1}{2a^{2}}\right),
\end{align*}
and therefore we obtain
\begin{displaymath}
x\in \mathbb{R}^{n}\setminus B_{d}\left(x_{1},\left(1-\frac{1}{2a^{2}}\right)d\left(x,y\right)\right).
\end{displaymath}
By the equivalence between $X$-balls and Carnot-Carathe\'odory balls, see (\ref{X-balls}), we have
\begin{eqnarray*}
\mathbb{R}^{n}\setminus B_{d}\left(x_{1},\left(1-\frac{1}{2a^{2}}\right)d\left(x,y\right)\right)&\subset& \mathbb{R}^{n}\setminus B\left(x_{1},\frac{1}{a}\left(1-\frac{1}{2a^{2}}\right)d\left(x,y\right)\right)\\
&\subset& \mathbb{R}^{n}\setminus \overline{B\left(x_{1},\left(a+3\right)r\right)},
\end{eqnarray*}
given that $a>1$.
Consider the cylinder 
\begin{displaymath}
C_{r}=B\left(x_{1},(a+3)r\right)\times [s,t).
\end{displaymath}
For $z\in \overline{B}\left(x_{1},(a+3)r\right)$,
\begin{equation}
\label{zB}
d\left(x,z\right)\geq d\left(x,x_{1}\right)-d\left(z,x_{1}\right)\geq d\left(x,y\right)\left(1-\frac{1}{2a^{2}}-\frac{1}{2a}\right).
\end{equation}
Define 
\begin{equation*}
w\left(z,\tau\right)\doteq C\left|\Lambda\left(x,\sqrt{t-s}\right)\right|\exp\left(\frac{M\beta^{2}_{2}d\left(x,y\right)^{2}}{t-s}\right)G\left(x,t;z,\tau\right),
\end{equation*}
where $\beta_{2}=\left(1-\frac{1}{2a^{2}}-\frac{1}{2a}\right)$ and $C=C(X,\Omega)$ is a suitable positive constant that will be chosen in a moment. 
Consequently, since $x\in\mathbb{R}^n\setminus \overline{B(x_1,(a+3)r)}$, we have from (\ref{pol}), the Gaussian bounds in (\ref{Gauss-Green}) and the estimate (\ref{zB}) that we can find a positive constant    $C=C(X,\Omega)$ such that
\begin{displaymath}
w\left(z,\tau\right)\leq C\frac{\left|\Lambda\left(x,\sqrt{t-s}\right)\right|}{\left|B(x,\sqrt{t-\tau})\right|}\exp\left(Md\left(x,z\right)^{2}\left(\frac{1}{t-s}-\frac{1}{t-\tau}\right)\right)\leq 1, 
\end{displaymath}
with  $(z,\tau)\in \left(\partial B(x_1,(a+3)r)\cap \Omega\right)\times [s,t)$. Furthermore,  $w\equiv 0$ on $S_T\cap C_r$ and 
\[
\lim_{\tau\nearrow t}w(z,\tau)=0 \quad \text{with $z\in B(x_1,(a+3)r)\cap D_T$}.
\]
Notice that we also have $L^*w=0$. Let $D^r\doteq C_r\cap D_T$ and  consider the solution $H^{D^r}_{\varphi}$ of the adjoint heat equation $L^*$,  where $\varphi\in C(\partial^*_p D^r)$, $0\leq \varphi \leq 1$, is such that
\[
\varphi \equiv 0\quad \text{on $\left(\left(B(x_1,(1+a)r)\cap \partial\Omega\right)\times(s,t)\right)\cup\left(\left(B(x_1,(1+a)r)\cap \partial\Omega\right)\times\{t\}\right)$.}
\]
Moreover, $\varphi\equiv 1$ on 
\[
\left(\partial B(x_1,(a+3)r)\cap \Omega\right)\times(s,t).
\]
Observe that $B(x_1,(a+3)r)\cap\Omega$ satisfies the uniform $X-$ball condition at $x_0\in\Omega$.  By the $L^*-$version of Lemma \ref{linear-growth} we conclude that for $(z,\tau) \in D^r$,
\[
H^{D^r}_{\varphi}(z,\tau)\leq C\frac{d(z,\partial\Omega)}{r}.
\]
Since $w\leq H^{D^r}_{\varphi}$ in $D^r$, the result follows with $M_1=M\beta^2_2$.
\end{proof}

Recall that the main goal of this section is to estimate the horizontal gradient of the Green function up to the lateral boundary. In order to establish this estimate we state first the following theorem from \cite{DG1}. Before stating the result, we introduce some notation. For $(x,t)\in \mathbb{R}^n$, we consider the following parabolic cylinders
\[
Q_r(x,t)=B_d(x,r)\times (t,t+r^2)\quad Q^+_r(x,t)=B_d(x,r)\times (t+\frac{r^2}{4},t+\frac{r^2}{2}). 
\]
\begin{thm}
Let $D\subset\mathbb{R}^{n+1}$ be an open subset and suppose that $u\geq 0$ solves $L^{\ast}u=0$. There exists $R_{1}>0$ depending on D and X such that for every $(x_{0},t_0)\in D$ and $0<r\leq R_{1}$ for which $\overline{Q}_r(x_{0},t_0)\subset D$, one has for any $s,k\in \mathbb{N}$,
for some constant $C=C\left(D,X,s,k\right)>0$, 
\begin{displaymath}
\sup_{Q^{+}_r(x_{0},t_0)}\left|\frac{\partial^{k}}{\partial t^{k}}X_{j_{1}}X_{j_{2}}...X_{j_{s}} u\right|\leq \frac{C}{r^{s+2k}}u\left(x_{0},t_0\right).
\end{displaymath}
\label{Int Cauchy}
In the above estimate, for every i=1,\ldots,s, the index $j_{i}$ runs in the set $\left\{1,\ldots,m\right\}$. 
\end{thm}
 
\noindent We can now prove our main estimate. 

\begin{thm}
\label{HorBound-G}
Asume the uniform $X$-ball condition for $\Omega$. There exist constants $C,M_1>0$, depending only on $X$ and $\Omega$, such that
\begin{eqnarray*}
\left|XG\left(x,t;y,s\right)\right|& \leq& \frac{C}{\left(d\left(x,y\right) \wedge R_{0}\right)}\times\frac{d(y,\partial\Omega)}{d(y,\partial\Omega)\wedge\sqrt{t-s}\wedge R_1}\\
&&\times \frac{1}{\left|B_{d}\left(x,\sqrt{t-s}\right)\right|}\exp\left(-\frac{M_1d\left(x,y\right)^{2}}{t-s}\right),\nonumber
\end{eqnarray*}
where $x\neq y\in \Omega$ and $0<s<t<T$. Here, $R_0$ is the constant in Definition \ref{Def-X-ball} and $R_1$ is the constant in Theorem \ref{Int Cauchy}.
\end{thm} 

\begin{proof}
Notice that the Green function, $G\left(x,t;y,s\right)$, is well defined for  $0 <s<t<T$. Let $0<r<R_1$ be as in Theorem \ref{Int Cauchy} such that $\left(x,t\right)\notin \overline{Q}_{r}\left(y,s-r^2/16\right)\subset D_T$. Notice that 
\[
\left(y,s\right)\in\overline{Q}^{+}_{r/2}\left(y,s-r^2/16\right).
\]
If we apply Theorem \ref{Int Cauchy} to the function $G(x,t;\cdot,\cdot)$ at the point $(y,s-r^2/16)$ we obtain
\begin{displaymath}
\left|XG\left(x,t;y,s\right)\right|\leq \frac{C}{r}G\left(x,t;y,s-r^{2}/16\right).
\end{displaymath}
Choose 
\begin{displaymath}
r=\min\left(2\sqrt{\frac{t-s}{15}},\frac{d\left(y,\partial\Omega\right)}{2},\frac{R_1}{2}\right),
\end{displaymath}
and then invoke  Theorem \ref{Bound-Est-G} to obtain
\begin{eqnarray*}
\left|XG\left(x,t;y,s\right)\right|& \leq &\frac{Cd\left(y,\partial \Omega\right)}{r\left(d\left(x,y\right)\wedge R_{0}\right)} \times \\
&& \frac{1}{\left|B_{d}\left(x,\sqrt{t-s+r^{2}/16}\right)\right|}\exp\left(-\frac{Md\left(x,y\right)^{2}}{t-s+r^{2}/16}\right).
\end{eqnarray*}
Observe that $t-s\leq t-s+r^{2}/16\leq 2\left(t-s\right)$, therefore
\begin{align*}
\left|XG\left(x,t;y,s\right)\right|& \leq \frac{Cd\left(y,\partial \Omega\right)}{r\left(d\left(x,y\right)\wedge R_{0}\right)\left|B_{d}\left(x,\sqrt{t-s}\right)\right|}\nonumber\exp\left(-\frac{M_1d\left(x,y\right)^{2}}{t-s}\right)\nonumber
\end{align*}
From this we easily obtain the desired conclusion in Theorem \ref{HorBound-G}, with $M_1=2M$.
\end{proof}

\begin{corollary}
\label{HeatBall-N-Sigma}
Let $D_T=\Omega\times\left(0,T\right)$, where $\Omega\subset\R^{n}$ is a bounded open set. If $\Om$ satisfies the outer $X$-ball condition in $V$, where $V$ is a neighborhood  of $\Sigma$, then for any $(x_{0},t_{0})\in D_T$ and every open neighborhood $U$ of $\pa\Omega$ such that $x_{0}\notin \overline{U}$, one has 
\[
\left\|G\left(x_{0},t_{0};\cdot,\cdot\right)\right\|_{L^{\infty}(U\times (0,t_{0}))} + \left\|XG\left(x_{0},t_{0};\cdot,\cdot\right)\right\|_{L^{\infty}(U\times (0,t_{0}))} \leq C(x_{0},t_{0},\Omega,V,U,X).
\]
\end{corollary}

\section{Representation formula for caloric functions}

In this section we establish a representation formula for smooth domains that satisfy the outer $X$-ball condition. Consider a cylinder $D_T=\Omega\times(0,T)$, where $\Omega\subset\Rn$ is a bounded open set regular for the Dirichlet problem, and denote by $G(x,t;y,s)$ the Green function for $L$ and $D_T$.  Fix a point $\left(x,t\right)\in D_T$ and consider a $C^{\infty}$ domain $U\subset \overline{U}\subset\Omega$ containing $x$, and we let $0<s<t<T$. For any $u,v\in C^{\infty}(D_T)$ we obtain from the divergence theorem 
\begin{eqnarray*}
\int_{U\times\left(0,s\right)}{\left(vLu-uL^{*}v\right)dyds}&= & \sum^{m}_{j=1}{\int^{s}_{0}{\int_{\partial U}{\left[uX_{j}v-vX_{j}u\right]\left\langle X_{j},\nu_{y}\right\rangle d\sigma}}}\\
& &-\int_{U}{v\left(y,s\right)u\left(y,s\right)dy}\boldsymbol{+}\int_{U}{v\left(y,0\right)u\left(y,0\right)dy}\nonumber
\end{eqnarray*}
Recall that $G\left(x,t;y,s\right)=p\left(x,t;y,s\right)-h_{\left(y,s\right)}\left(x,t\right)$, where $h_{(y,s)}$ is the unique $L$-parabolic function with boundary values $p(\cdot,\cdot\ ;y,s)$. The function $G$ is continuous in any relatively compact subdomain of $\overline{D_T}\setminus \left\{\left(x,t\right)\right\}$.  If we let $v\left(y,s\right)=G\left(x,t;y,s\right)$ in the above equation, keeping in mind that $L^{*}G(x,t;\cdot,\cdot)=0$, we obtain for $0<s<t$
\begin{eqnarray*}
\int_{U\times\left(0,s\right)}{G(x,t;y,\tau)Lu(y,\tau)dyd\tau}&= & \sum^{m}_{j=1}{\int^{s}_{0}{\int_{\partial U}{\left[uX_{j}G-GX_{j}u\right]\left\langle X_{j},\nu_{y}\right\rangle d\sigma}}}\\
& &-\int_{U}{G\left(x,t;y,s\right)u\left(y,s\right)dy}\\ 
&&+\int_{U}{G\left(x,t;y,0\right)u\left(y,0\right)dy}.\nonumber
\end{eqnarray*}

\noindent By Theorem \ref{Green-th} we have
\[
u\left(x,t\right)=\lim_{s\nearrow t}\int_{U}{G\left(x,t;y,s\right)u\left(y,s\right)dy}.
\]
We thus conclude
\begin{eqnarray*}
u\left(x,t\right)&=&-\int_{U\times\left(0,t\right)}{G(x,t;y,s)Lu(y,s)dyds}+\int_{U}{G\left(x,t;y,0\right)u\left(y,0\right)dy}\\ 
&&+\sum^{m}_{j=1}{\int^{t}_{0}{\int_{\partial U}{\left[uX_{j}G-GX_{j}u\right]\left\langle X_{j},\nu_{y}\right\rangle d\sigma ds}}}.\nonumber
\end{eqnarray*}

\begin{definition}
\label{Def-W}
Given a bounded open set $\Omega\subset\overline{\Omega}\subset\R^{n}$ of class $C^{1}$, at every point $y\in\pa\Omega$ we let
\[
N^{X}\left(y\right)=\left(\left\langle \nu\left(y\right),X_{1}\left(y\right)\right\rangle,\ldots,\left\langle \nu\left(y\right),X_{m}\left(y\right)\right\rangle\right),
\]
where $\nu(y)$ is the outer unit normal to $\Omega$ in $y$. We also set 
\[
W\left(y\right)=\left|N^{X}\left(y\right)\right|=\sqrt{\sum^{m}_{j=1}\left\langle \nu\left(y\right),X_{j}\left(y\right)\right\rangle^{2}}.
\]
If $y\in\pa\Omega\setminus\Sigma$, we set
\[
\nu^{X}\left(y\right)=\frac{N^{X}\left(y\right)}{\left|N^{X}\left(y\right)\right|}.
\]
One has $\left|\nu^{X}\left(y\right)\right|=1$ for every $y\in\pa\Omega\setminus\Sigma$.
\end{definition}
We note explicitly that one has for the characteristic set $\Sigma$ of $\Omega$
\[
\Sigma=\left\{y\in\pa\Omega:W\left(y\right)=0\right\}.
\]
\begin{proposition}
\label{Rep-U}
Let $D_T=\Omega\times(0,T)$, where $\Omega\subset\R^{n}$ is a bounded open set,  with positive Green function $G$. Consider a $C^{1}$ domain $U\subset \overline{U}\subset\Omega$. For any $u\in C^{\infty}(D_T)$, and every $(x,t)\in U\times (0,T)$ one has
\begin{eqnarray*}
u\left(x,t\right)&=&-\int_{U\times\left(0,t\right)}{G\left(x,t;y,s\right)Lu\left(y,s\right)dyds}+\int_{U}{u\left(y,0\right)G\left(x,t;y,0\right)dy}\\ &-&\int^{t}_{0}{\int_{\partial U}{G\left(x,t;y,s\right)\left\langle Xu,N^{X}\left(y,s\right)\right\rangle d\sigma\left(y\right)ds}}\\
&+&\int^{t}_{0}{\int_{\partial U}{u\left(y,s\right)\left\langle XG\left(x,t;y,s\right),N^{X}\left(y,s\right)\right\rangle d\sigma\left(y\right)ds}}.
\end{eqnarray*}
If moreover $Lu=0$ in $D_T$, then
\begin{eqnarray*}
u\left(x,t\right)&=&-\int^{t}_{0}{\int_{\partial U}{G\left(x,t;y,s\right)\left\langle Xu,N^{X}\left(y\right)\right\rangle d\sigma\left(y\right)ds}}+\int_{U}{u\left(y,0\right)G\left(x,t;y,0\right)dy}\\
&+&\int^{t}_{0}{\int_{\partial U}{u\left(y,s\right)\left\langle XG\left(x,t;y,s\right),N^{X}\left(y\right)\right\rangle d\sigma\left(y\right)ds}}
\end{eqnarray*}
In particular, the latter equality gives for every $(x,t)\in D_T$ 
\begin{eqnarray*}
1&=&\int^{t}_{0}{\int_{\partial U}{\left\langle XG\left(x,t;y,s\right),N^{X}\left(y\right)\right\rangle d\sigma\left(y\right)ds}}\\
&+&\int_{U}{G\left(x,t;y,0\right)dy}
\end{eqnarray*}
\end{proposition}
The following result due to Derridj, see Theorem 1 in \cite{De1}, will be important in the sequel.
\begin{thm}
\label{Sigma0}
Let $\Omega\subset\R^{n}$ be a $C^\infty$ domain. If $\Sigma$ denotes its characteristic set, then $\sigma(\Sigma)=0$. As a consequence, if we consider the cylinder $D_T = \Om \times (0,T)$, with lateral boundary $S_T = \pa \Om \times (0,T)$, then the set $\Sigma \times (0,T) \subset S_T$ has vanishing $d\sigma \times dt$ measure.
\end{thm}

\begin{definition}
With the notation of Definition \ref{Def-W}, for every $((x,t),(y,s))\in D_T\times(\pa\Omega\setminus\Sigma\times \left(0,T\right) )$ we define
\begin{equation}
\label{Def-P}
P(x,t;y,s)=
\left\{
\begin{array}{ll}
\left\langle XG\left(x,t;y,s\right),N^{X}\left(y\right)\right\rangle,\ \  & 0<s<t,\\
0,\ \ \  & s\geq t. 
\end{array} \right.
\end{equation}
We also define 
\begin{equation}
\label{Def-K}
K(x,t;y,s)=\frac{P(x,t;y,s)}{W(y)}.
\end{equation}
We extend the definition of $P$ and $K$ to all $D_T\times S_T$ by letting $P(x,t;y,s) = K(x,t;y,s) = 0$ for any $(x,t) \in D_T$ and $(y,s)\in\Sigma\times(0,T)$.  According to Theorem \ref{Sigma0} the extended functions coincide $d\sigma\times dt$-a.e. with the kernels in (\ref{Def-P}), (\ref{Def-K}).
\end{definition}
\noindent For $(x,t;y,s)\in D\times(\pa\Omega\setminus\Sigma\times \left(0,T\right) )$, with $0<s<t$ we have
\begin{equation}
\label{Bound-P-K}
P(x,t;y,s)\leq W(y)\left|XG(x,t;y,s)\right|,\quad K(x,t;y,s)\leq \left|XG(x,t;y,s)\right|.
\end{equation}
We introduce a new measure on $\pa\Omega$ by letting
\begin{equation}
\label{New-meas}
d\sigma_{X}=Wd\sigma.
\end{equation}
Observe that since we are assuming that $\Omega\in C^{\infty}$ the density $W$ is smooth and bounded on $\pa\Omega$ and therefore (\ref{New-meas}) implies that $d\sigma_{X}<< d\sigma$. In view of this observation Theorem \ref{Sigma0} implies $\sigma_{X}(\Sigma)=0$.
\begin{proposition}
\label{PandK=1}
Assume that $D_T=\Omega\times (0,T)$, where $\Omega\subset\R^{n}$ is a bounded $C^{\infty}$ domain that satisfies the uniform outer $X$-ball condition in a neighborhood of its characteristic set $\Sigma$. For every $(x,t)\in D_T$ we have
\begin{eqnarray*}
1&=&\int^{t}_{0}{\int_{\partial \Omega}P(x,t;y,s)d\sigma\left(y\right)ds}+\int_{\Omega}{G\left(x,t;y,0\right)dy}\\
&=&\int^{t}_{0}{\int_{\partial \Omega}K(x,t;y,s)d\sigma_{X}\left(y\right)ds}+\int_{\Omega}{G\left(x,t;y,0\right)dy}
\end{eqnarray*} 
\end{proposition}
\begin{proof}
Fix $(x,t)\in D_T$. In view of Theorem \ref{Sigma0} we can choose an exhaustion of $\Omega$ with a family of $C^{\infty}$ connected open sets $\Omega_{k}\subset\overline{\Omega}_{k}\subset\overline{\Omega}$, with $\Omega_{k}\nearrow\Omega$ as $k\rightarrow\infty$, such that $\pa\Omega_{k}=\Gamma^{1}_{k}\cup\Gamma^{2}_{k}$, with $\Gamma^{1}_{k}\subset\pa\Omega\setminus\Sigma$, $\Gamma^{1}_{k}\nearrow\pa\Omega$, $\sigma(\Gamma^{2}_{k})\rightarrow 0$. By Proposition \ref{Rep-U} we obtain for every $k\in\mathbb{N}$
\begin{eqnarray}
\label{Rep-Om-k}
1&=&\int^{t}_{0}{\int_{\partial \Omega_{k}}P(x,t;y,s)d\sigma\left(y\right)ds}+\int_{\Omega_{k}}{G\left(x,t;y,0\right)dy}\\
&=& \int^{t}_{0}{\int_{ \Gamma^{1}_{k}}P(x,t;y,s)d\sigma\left(y\right)ds}+\int^{t}_{0}{\int_{ \Gamma^{2}_{k}}P(x,t;y,s)d\sigma\left(y\right)ds}\nonumber\\
&+&\int_{\Omega_{k}}{G\left(x,t;y,0\right)dy}.\nonumber
\end{eqnarray}
We now pass to the limit as $k\rightarrow\infty$ in the above integrals. Using Corollary \ref{HeatBall-N-Sigma} and $\sigma(\Gamma^{2}_{k})\rightarrow 0$, we infer
\[
\lim_{k\rightarrow \infty}\int^{t}_{0}{\int_{ \Gamma^{2}_{k}}\left\langle XG\left(x,t;y,s\right),N^{X}\left(y\right)\right\rangle d\sigma\left(y\right)ds}=0.
\]
The fact that $\Gamma^{1}_{k}\nearrow\pa\Omega$ and once again Corollary \ref{HeatBall-N-Sigma}, allow to use the dominated convergence theorem, and obtain
\begin{eqnarray*}
\lim_{k\rightarrow \infty}\int^{t}_{0}{\int_{ \Gamma^{1}_{k}}\left\langle XG\left(x,t;y,s\right),N^{X}\left(y\right)\right\rangle d\sigma\left(y\right)ds}=\\
\int^{t}_{0}{\int_{ \pa\Omega}\left\langle XG\left(x,t;y,s\right),N^{X}\left(y\right)\right\rangle d\sigma\left(y\right)ds}.
\end{eqnarray*}
In conclusion, we have found
\begin{equation*}
1=\int^{t}_{0}{\int_{ \pa\Omega}\left\langle XG\left(x,t;y,s\right),N^{X}\left(y\right)\right\rangle d\sigma\left(y\right)ds}+\int_{\Omega}G\left(x,t;y,0\right)dy.
\end{equation*}
The first equality is thus established. In order to prove the second identity we rewrite (\ref{Rep-Om-k}) as follows
\begin{eqnarray*}
1&=& \int^{t}_{0}{\int_{ \Gamma^{1}_{k}}K(x,t;y,s)d\sigma_{X}\left(y\right)ds}+\int^{t}_{0}{\int_{ \Gamma^{2}_{k}}P(x,t;y,s)d\sigma\left(y\right)ds}\nonumber\\
&+&\int_{\Omega_{k}}{G\left(x,t;y,0\right)dy}.\nonumber
\end{eqnarray*}
Since we have observed $d\sigma_{X}<<d\sigma$, in view of the second estimate $K(x,t;y,s)\leq \left|XG(x,t;y,s)\right|$ in (\ref{Bound-P-K}), we can again use Corollary \ref{HeatBall-N-Sigma} and dominated convergence (with respect to $\sigma_{X}$) to conclude that 
\[
\lim_{k\rightarrow\infty}\int^{t}_{0}{\int_{ \Gamma^{1}_{k}}K(x,t;y,s)d\sigma_{X}\left(y\right)ds}=\int^{t}_{0}{\int_{ \pa\Omega}K(x,t;y,s)d\sigma_{X}\left(y\right)ds}.
\]
This completes the proof.
\end{proof}
\begin{thm}
\label{H-P-K}
Let $D_T=\Omega\times (0,T)$ satisfy the assumptions in Proposition \ref{PandK=1}. If $\varphi\in C^{\infty}_0(\partial_p D_T )$, with $\mathrm{supp}\ \varphi  \subset S_T$, assumes a single constant value in a neighborhood of $\Sigma\times (0,T)$, then 
\begin{equation}\label{gb}
||XH^{D_T}_{\varphi}||_{L^\infty(D_T)} < \infty.
\end{equation}
Furthermore, if for $\varphi\in C(\pa_p D_T)$, with $\mathrm{supp}\ \varphi\subset S_T$, the function $H^{D_T}_\varphi$ satisfies \eqref{gb}, then for every $(x,t)\in D_T$ one has
\begin{eqnarray*}
H^{D_T}_{\varphi}\left(x,t\right)&=&\int^{t}_{0}{\int_{\partial \Omega}P(x,t;y,s)\varphi\left(y,s\right)d\sigma\left(y\right)ds}\\
&=&\int^{t}_{0}{\int_{\partial \Omega}K(x,t;y,s)\varphi\left(y,s\right)d\sigma_{X}\left(y\right)ds}.
\end{eqnarray*} 
\end{thm}

\begin{proof}
We start with the proof of the regularity result. Let $\varphi$ be as in the first part of the statement. Consider $U\times (0,T)$, where $U$ is a neighborhood of $\Sigma$, in which the function $\varphi$ is constant and along which the domain $D_T$ satisfies the uniform outer $X-$ball condition. Let us assume that $U=\cup_{P\in\Sigma}B(P,\varepsilon)$, for some $\varepsilon=\varepsilon(U,X)>0$. If we denote by $R_{0}$ the constant involved in
the definition of $X$-ball (see Definition \ref{X-ball} above), then we can always select a smaller
constant so that $\varepsilon=2R_{0}$. In view of Proposition \ref{PandK=1} we
can assume without loss of generality that $\varphi$ vanishes in a neighborhood of $\Sigma\times (0,T)$ and $\underset{\pa_p D}{\max}\left|\varphi\right|=1$. We want to show that the horizontal gradient of $H^{D_T}_{\varphi}$ is in $L^{\infty}$ in such neighborhood. Fix $x_{0}\in\Sigma$, and $0<r<R_{0}$, where $R_{0}$ is as in Definition \ref{X-ball}. Suppose that $\varphi\equiv 0$ on $\left(B(x_0,\varepsilon)\cap\Omega\right)\times (0,T)$. Moreover, since $\varphi\equiv 0$ on $\left(B(x_0,\varepsilon)\cap \overline{\Omega}\right)\times \{0\}$, we  conclude from Lemma \ref{linear-growth} that there exist $C=C(X,\Omega)>0$ such that for every $(y,s)\in \left(B(x_0,\varepsilon)\cap\Omega\right)\times (0,T)$, 
\begin{equation*}
\left|H^{D_T}_{\varphi}\left(y,s\right)\right|\leq C\frac{d\left(x_{0},y\right)}{r}.
\end{equation*}
Suppose that $x\in B_{d}\left(x_{0},r/2\right)\cap\Omega$ and consider the ball $B_{d}(x,\tau)$, where $\tau=\min\left(\frac{d(x,\pa\Omega)}{4},\frac{\sqrt{T-s}}{4}\right)$, with $0<s<T$. Theorem \ref{Int Cauchy} implies
\begin{equation}
\label{XH-Bd}
\left|XH^{D_T}_{\varphi}(x,s) \right|\leq \frac{C}{d(x,\pa\Omega)}H^{D_T}_{\varphi}\big(x,s+\frac{\tau^{2}}{16}\big).
\end{equation}
Let $P\in\partial\Omega$ such that $d(x,P)=d(x,\pa\Omega)$. Observe that $d(P,\Sigma)\leq d(P,x_{0})+d(x,x_{0})\leq 2d(x,x_{0})\leq R_{0}=\varepsilon/2$. In particular, we can apply once more Lemma \ref{linear-growth}, and obtain 
\[
\left|H^{D}_{\varphi}\big(x,s+\frac{\tau^{2}}{16}\big)\right|\leq C\frac{d\left(x,P\right)}{r}= C\frac{d\left(x,\pa\Omega\right)}{r}.
\]
The latter inequality and (\ref{XH-Bd}) imply
\[
\left|XH^{D}_{\varphi}(x,s) \right|\leq \frac{C}{r}.
\]
This proves that $\left|XH^{D}_{\varphi} \right|\in L^{\infty}((B(x_{0},r/2)\cap\Omega) \times (0,T) )$, from which \eqref{gb} follows. To establish the second part of the theorem, we take a function $\varphi\in C(\pa_{p} D_T)$ with $\mathrm{supp}\ \varphi\subset S_T$ for which \eqref{gb} hold. We fix $(x,t)\in D_T$, and consider a sequence of $C^{\infty}$ domains $\Omega_{k}$ as in the proof of Proposition \ref{PandK=1}. Proposition \ref{Rep-U} gives
\begin{eqnarray*}
H^{D}_{\varphi}\left(x,t\right)&=&\int^{t}_{0}{\int_{\partial \Omega_{k}}{G\left(x,t;y,s\right)\left\langle \left(XH^{D}_{\varphi}\right)\left(y,s\right),N^{X}\left(y\right)\right\rangle d\sigma\left(y\right)ds}}\\
&-&\int^{t}_{0}{\int_{\partial \Omega_{k}}{H^{D}_{\varphi}\left(y,s\right)\left\langle XG\left(x,t;y,s\right),N^{X}\left(y\right)\right\rangle d\sigma\left(y\right)ds}}\\
&+&\int_{\Omega_{k}}{H^{D}_{\varphi}\left(y,0\right)G\left(x,t;y,0\right)dy}.
\end{eqnarray*}
At this point the conclusion follows along the lines of the proof of Proposition \ref{PandK=1}.
\end{proof}

\begin{proposition}
\label{Positive-P-K}
Let $\Omega$ be a $C^{\infty}$ domain and consider $D_T=\Omega\times(0,T)$. 
\begin{itemize}
\item[i)] If $\Omega$ satisfies the uniform outer $X$-ball condition in a neighborhood of $\Sigma$, then $P(x,t;y,s)\geq 0$ and $K(x,t;y,s)\geq 0$ for each $((x,t),(y,s))\in D_T\times S_T$; 
\item[ii)] If $\Omega$ satisfies the uniform outer $X$-ball condition, then there is a constant $C_{\Omega}>0$ such that for $((x,t),(y,s))\in D_T\times S_T$,
\end{itemize} 
\begin{align*}
0\leq P(x,t;y,s)\leq \frac{C_{\Omega}}{\left(d\left(x,y\right)^{2}\wedge \sqrt{t-s}\wedge R_{0}\right)}\frac{1}{\left|B_{d}\left(x,\sqrt{t-s}\right)\right|}\exp\left(-\frac{M_1d\left(x,y\right)^{2}}{\left(t-s\right)}\right),\\
0\leq K(x,t;y,s)\leq \frac{C_{\Omega}}{\left(d\left(x,y\right)^{2}\wedge \sqrt{t-s}\wedge R_{0}\right)}\frac{1}{\left|B_{d}\left(x,\sqrt{t-s}\right)\right|}\exp\left(-\frac{M_1d\left(x,y\right)^{2}}{\left(t-s\right)}\right).
\end{align*}
\end{proposition}
\begin{proof}
We start with the proof of part \emph{i)}. For $(y_0,s_0)\in S_T$ and $r>0$, we consider the cylinder
\[
C^X_r(y_0,s_0)\doteq B(y_0,r)\times (s_0-r^2,s_0+r^2). 
\]
Suppose that for some $(x_0,t_0)\in D_T$ and $(y_{0},s_{0})\in S_T$, with $0<s_0<t_0$, we have $P(x_0,t_0;y_{0},s_{0})=\alpha<0$. This implies that $y_{0}\notin\Sigma$. By the smoothness away from characteristic points, see Theorem \ref{Derr}, there exists a sufficiently small $r>0$ such that $P(x_0,t_0;y,s)\leq \alpha/2$ for every $(y,s)\in C^{X}_{2r}(y_{0},s_{0})\cap\pa_{p}D_T$. We can  assume that $2r< d(y_{0},\Sigma)$. 

We now choose $\varphi\in C^{\infty}(\pa_{p} D_T)$ such that $0\leq \varphi\leq 1$, $\varphi\equiv 1$ on $C^{X}_{r}(y_{0},s_{0})\cap\pa_{p}D_T$ and $\varphi\equiv 0$ outside $C^{X}_{3r/2}(y_{0},s_{0})\cap\pa_{p}D_T$. The maximum principle implies $H^{D_T}_{\varphi}\geq 0$ in $D_T$, where $H^{D_T}_{\varphi}$ is the Perron-Wiener-Brelot solution to Dirichlet problem (\ref{Dirichlet-prob2}) with boundary data $\varphi$. By the Harnack inequality or the strong maximum principle, we must have $H^{D_T}_{\varphi}(x,t)>0$ in $\Omega \times (s_0-r^2/16,T)$. On the other hand, Theorem \ref{H-P-K} gives
\[
H^{D_T}_{\varphi}(x,t)\leq \frac{\alpha}{2}\int^{t}_{0}\int_{B(y_{0},3r/2)\cap\pa\Omega}\varphi(y,s) d\sigma(y)ds\leq 0,
\]
 The proof of part \emph{ii)} is an inmediate consequence of (\ref{Bound-P-K}) and of Theorem \ref{HorBound-G}. The estimate for $K(x,t;y,s)$ follows from (\ref{Def-K}) and from the one for $P(x,t;y,s)$.
\end{proof}

\begin{remark}
\label{Greenbottom}
It is not difficult to see that by the properties of the Green function  we have for every Borel set $E\subset \Omega$ 
\[
\omega^{(x,t)}(E)=\int_EG(x,t;y,0)dy.
\]
\end{remark}

We now fix $(x,t)\in D_T$. For every $\sigma-$measurable $E\subset S_T$ we set 
\[
\nu^{(x,t)}(E)=\int_{E}K(x,t;y,s)d\sigma_{X}(y)ds.
\]
According to Proposition \ref{Positive-P-K}, $d\nu^{(x,t)}$ defines a Borel measure on $S_T$. The following theorem, a consequence of Theorems \ref{Sigma0} and \ref{H-P-K}, is the main result of this section. 

\begin{thm}
\label{Main-P-K}
Let $\Omega\subset\R^{n}$ be a $C^{\infty}$ domain possessing the uniform outer X-ball  condition in a neighborhood of the characteristic set $\Sigma$. Consider $D_T=\Omega\times(0,T)$, with $0< T< \infty$. For every $(x,t)\in D_T$, we have $\omega^{(x,t)}=\nu^{(x,t)}$, i.e., for every $\varphi\in C(\partial_p D_T)$ with $\mathrm{supp}\ \varphi \subset S_T$, one has
\begin{eqnarray*}
H^{D_T}_{\varphi}\left(x,t\right)&=&\int^{t}_{0}{\int_{\partial \Omega}P(x,t;y,s)\varphi\left(y,s\right)d\sigma\left(y\right)ds}\\
&=&\int^{t}_{0}{\int_{\partial \Omega}K(x,t;y,s)\varphi\left(y,s\right)d\sigma_{X}\left(y\right)ds}.
\end{eqnarray*} 
 In particular, $d\omega^{(x,t)}$ is absolutely continuous with respect to $d\sigma_{X}\times dt$ and $d\sigma\times ds$ on $S_T = \pa\Omega\times(0,T)$, and for every $((x,t)(y,s))\in D_T\times S_T$ one has 
\begin{equation}
\label{Omega-sigma}
\frac{d\omega^{(x,t)}}{d\sigma_{X}\times dt}(y,s)=K(x,t;y,s), \quad \frac{d\omega^{(x,t)}}{d\sigma\times dt}(y,s)=P(x,t;y,s).
\end{equation}
\end{thm}

\begin{proof}
We start with the proof of \eqref{Omega-sigma}. Let $\varepsilon>0$ and $(x,t)\in D_T$. From Theorem \ref{Sigma0} and the estimates for $K$, we have that there exists open sets $\Sigma_{\varepsilon}, U_{\varepsilon}$ such that $\Sigma\subset\Sigma_{\varepsilon}\subset\overline{\Sigma_{\varepsilon}}\subset U_{\varepsilon}$ and $\nu^{(x,t)}(U_{\varepsilon}\times (0,T))<\varepsilon/2$.  Consider a function $\varphi\in C^{\infty}_0(\partial\Omega)$ with $0\leq\varphi\leq 1$ and $\varphi=1$ on $U_{\varepsilon}$, such that the function $\varphi:\overline{D_T}\rightarrow\R$, defined by slightly abusing the notation as $\varphi(y,s)=\varphi(y)$ for every $(y,s)\in\overline{D_T}$, satisfies $\nu^{(x,t)}(\mathrm{supp}\ \varphi)\leq \frac{3}{4}\varepsilon$. Moreover, for every $k\in\mathbb{N}$ let $\varphi_k:[0,T]\rightarrow\R$ be a smooth function such that $0\leq \varphi_k\leq 1$, $\varphi_k\equiv1$ on $[k^{-1},T-k^{-1}]$ and $\varphi_k\equiv 0$ on $[0,(2k)^{-1}]\cup[T-(2k)^{-1},T]$. We then have
\begin{eqnarray*}
\omega^{(x,t)}\left(U_{\varepsilon}\times \left(\frac{1}{k},T-\frac{1}{k}\right)\right)&=&\int^{T-\frac{1}{k}}_\frac{1}{k}\int_{U_{\varepsilon}}d\omega^{(x,t)}(y,s)\leq \int^{T}_0\int_{\partial\Omega}\varphi(y,s)\varphi_n(s)d\omega^{(x,t)}(y,s)\\
&=& H^{D_T}_{\varphi\varphi_k}(x,t)\quad (\text{by Theorem \ref{H-P-K}})\\
&=&\int^t_0\int_{\partial\Omega}K(x,t;y,s)\varphi(y,s)\varphi_k(s)d\sigma_X(y)ds\\
&\leq&\int^t_0\int_{\partial\Omega}K(x,t;y,s)\varphi(y,s)d\sigma_X(y)ds\\
&\leq & \nu^{(x,t)}(\mathrm{supp}\ \varphi)\leq \frac{3}{4}\varepsilon.
\end{eqnarray*}
By letting $k\rightarrow \infty$ we find
\[
\omega^{(x,t)}\left(U_{\varepsilon}\times \left(0,T\right)\right)\leq \frac{3}{4}\varepsilon.
\]

 Let $F \subseteq S_T$ be a Borel set. If $F=S_T$ then \eqref{Omega-sigma} follows from Theorem \ref{PandK=1} and the Remark \ref{Greenbottom}.
 Let us thus assume that $F\subset S_T$.  Once again by the boundedness of the functions $K(\cdot;\cdot)$ and $W(\cdot)$, we can find open sets $E_{\varepsilon},F_{\varepsilon}\subset S_T$ such that $F\subset F_{\varepsilon}\subset \overline{F_{\varepsilon}}\subset E_{\varepsilon}$ and $\nu^{(x,t)}(E_{\epsilon}\setminus F)\leq \varepsilon/2$. Let now $\psi_0,\psi_1\in C^{\infty}(S_T)$ such that $0\leq\psi_0,\psi_1\leq 1$, and for which 
 \begin{eqnarray*}
 \psi_0&\equiv& 1 \ \text{on} \ S_T\setminus U_{\varepsilon}\times (0,T), \quad  \psi_0\equiv 0 \ \text{on}\ \Sigma_{\varepsilon}\times (0,T),\\
  \psi_1&\equiv& 1 \ \text{on} \ F, \quad  \psi_1\equiv 0 \ \text{on}\ S_T\setminus E_{\varepsilon}.
 \end{eqnarray*}
Then, we have
 \begin{eqnarray*}
 \omega^{(x,t)}\left(F\cap \left(\partial\Omega\times \left(\frac{1}{k},T-\frac{1}{k}\right)\right)\right)&\leq &\omega^{(x,t)}(U_{\varepsilon})+\omega^{(x,t)}\left(F\setminus U_{\varepsilon}\cap \left(\partial\Omega\times \left(\frac{1}{k},T-\frac{1}{k}\right)\right)\right)\\
 &\leq & \frac{3}{4}\varepsilon+\int_{S_T}\psi_0(y,s)\psi_1(y,s)\varphi_k(s)d\omega^{(x,t)}(y,s)\\
 &=&\frac{3}{4}\varepsilon+H^{D_T}_{\psi_0\psi_1\varphi_k}(x,t)\quad (\text{by Theorem \ref{H-P-K}})\\
 &=& \frac{3}{4}\varepsilon+\int^t_0\int_{\partial\Omega}\psi_0(y,s)\psi_1(y,s)\varphi_k(s)K(x,t;y,s)d\sigma_X(y)ds\\
  &\leq& \frac{3}{4}\varepsilon+\int^t_0\int_{\partial\Omega}\psi_0(y,s)\psi_1(y,s)K(x,t;y,s)d\sigma_X(y)ds\\
 &\leq &\frac{3}{4}\varepsilon+\nu^{(x,t)}(E_{\varepsilon})\\
 &\leq &\frac{3}{4}\varepsilon +\nu^{(x,t)}(F)+\nu^{(x,t)}(E_{\varepsilon}\setminus F)< \nu^{(x,t)}(F)+\frac{5}{4}\varepsilon.
 \end{eqnarray*}
 Since $\varepsilon>0$ is arbitrary, we conclude by letting $k\rightarrow\infty$ that $\omega^{(x,t)}(F)\leq \nu^{(x,t)}(F)$. The argument can be repeated with the set $E_{\varepsilon}\setminus F$ instead of $F$. Hence, we obtain $\omega^{(x,t)}(E_{\varepsilon}\setminus F)\leq \nu^{(x,t)}(E_{\varepsilon}\setminus F)$. From this we conclude,  by exchanging the role of $\omega^{(x,t)}$ and $\nu^{(x,t)}$ in the above computations, that $\nu^{(x,t)}(F)\leq\omega^{(x,t)}(F)$. This proves (\ref{Omega-sigma}).
 
 From (\ref{Omega-sigma}) the rest of the proof easily follows. By the definition of $L-$parabolic measure, we have for  $\varphi\in C(\partial_p D_T)$, with  $\mathrm{supp}\ \varphi \subset S_T$, that
 \[
 H^D_{\varphi}(x,t)=\int_{S_T}\varphi(y,s)d\omega^{(x,t)}(y,s).
 \]
 We conclude from (\ref{Omega-sigma}) that $d\omega^{(x,t)}(y,s)=K(x,t;y,s)d\sigma_X(y)ds$, and this completes the proof  of the theorem.
\end{proof}

\section{Reverse H\"older inequalities for the Poisson kernel}

This section is devoted to proving the main results of this paper, namely Theorems \ref{HolderK}, \ref{NonTanK},
\ref{HolderP} and \ref{NonTanP}. In the course of the proofs we will need some basic results about $NTA$ domains
from the papers \cite{FrG} and \cite{M1}, see also \cite{FGGMN}. We begin by recalling the relevant definitions.

\begin{definition}
We say that $\Omega$ is an non-tangentially accessible domain ($NTA$ domain)
if there exists $M$, $r_{0}>0$ for which:

\begin{enumerate}
\item (Interior corkscrew condition) For any $Q\in\partial\Omega$ and $r\leq
r_{0}$ there exists $A_{r}\left(Q\right)\in\Omega$ such that $\frac{r}{M}
\leq d\left(A_{r}\left(Q\right),Q\right)\leq r$ and $d\left(A_{r}\left(Q
\right),\partial \Omega\right)>\frac{r}{M}$.(This implies that $
B_{d}\left(A_{r}\left(Q\right),\frac{r}{2M}\right)$ is $\left(3M,X\right)$-nontangential.)

\item (Exterior corkscrew condition) $\Omega^{c} = \mathbb{R}
^n\setminus\Omega$ satisfies property (1).

\item (Harnack chain condition) For any $\ve>0$ and $x,y\in\Omega$ such
that $d\left(x,\partial \Omega\right)>\ve$, $d\left(y,\partial
\Omega\right)>\ve$, and $d\left(x,y\right)<2^{k}\ve$, there exists
a Harnack chain joining $x$ to $y$ of length $Mk$ and such that the diameter of
each ball is bounded from below by $M^{-1}\min\left\{d\left(x,\partial\Omega
\right),d\left(y,\partial\Omega\right)\right\}$.
\end{enumerate}
\label{NTA-def}
\end{definition}

\begin{lemma}\label{L:cork}
Let $\Omega\subset \R^{n}$ be an $NTA$ domain. Then, there exist positive constants $C,
R_1$, depending on the $NTA$ parameters of $\Omega$, such that for every
$y\in \pa \Om$ and every $0<r<R_1$ one has,
\[
C  |B_d(y,r)|\leq \min \{|\Om\cap B_d(y,r)|,|\Om^c \cap B_d(y,r)|\}
\leq C^{-1} |B_d(y,r)|.
\]
In particular, every $NTA$ domain has outer positive $d$-density and therefore, in view of Lemma \ref{Dirichlet0}, given $f\in C(\partial_p D_T)$,
there exists a unique solution $u\in\Gamma^{2}(D_T)\cap
C(D_T\cup\partial_p D_T)$ to the Dirichlet problem \eqref{dirp}.
In particular, $D_T$ is regular.
\end{lemma}

We now state a basic non-degeneracy property of the horizontal perimeter measure $d\sigma_X$. The proof can be found in Theorem 8.3 in \cite{CGN4}.
\begin{thm}
\label{Basic-ND}
Let $\Omega\subset\Rn$ be an $NTA$ domain of class $C^2$, then there exist $C^\star, R_1 > 0$,
depending on $\Om, X$ and on the $NTA$ parameters of $\Omega$, such that for every $y\in\pa\Omega$ and every $0<r<R_1$
\[
\sigma_X(\pa\Omega\cap B(y,r))\geq C^\star \frac{\left|B_d(y,r)\right|}{r}.
\]
In particular, $\sigma_X$ is lower $1$-Alfhors according to \cite{DGN2} and $\sigma_X(\pa\Omega\cap B(y,r))>0$.
\end{thm} 
\begin{corollary}
\label{Basic-ND1}
Let $\Omega\subset\Rn$ be an $NTA$ domain of class $C^2$ satisfying the upper $1-$Ahlfors assumption in iv) of Definition \ref{ADP}. Then, the measure $\sigma_X$ is $1$-Ahlfors, in the sense that there exist $A_1, R_1>0$ depending in the $NTA$ parameters of $\Omega$ and on $A>0$ in iv), such that for every $y\in\pa\Omega$, and every $0<r<R_1$, one has 
\begin{equation}
\label{Ahl-1}
A_1 \frac{\left|B_d(y,r)\right|}{r}\leq \sigma_X(\pa\Omega\cap B_d(y,r))\leq A^{-1}_1 \frac{\left|B_d(y,r)\right|}{r}.
\end{equation}
In particular, the measure $\sigma_X$ is doubling, i.e., there exists $C>0$ depending on $A_1$ and on the constant $C_1$ in (\ref{doubling}), such that
\begin{equation}
\label{Doubling-sigmaX}
\sigma_X(\pa\Omega\cap B_d(y,2r))\leq C\ \sigma_X(\pa\Omega\cap B_d(y,r)),
\end{equation}
for every $y\in\pa\Omega$ and $0<r<R_1$.
\end{corollary}
The following results from \cite{FrG}, \cite{FGGMN} and \cite{M1}, namely Lemma \ref{NonVa}, Theorem \ref{Doubling-Cal}, Theorem \ref{Globalcom} and Theorem \ref{Non-Max} below, play a fundamental role in the rest of the paper. All of these results assume that the cross-section $\Om$ of the cylinder $D_T$ is an $NTA-$domain with respect to the intrinsic distance of the system $X$. 

For $(Q,s)\in S$ and $r>0$, we define
\[
\overline{A}_r(Q,s)=(A_r(Q),s+2r^2),
\]
and 
\[
\Delta_r(Q,s)=\partial_p D_T\cap C_r(Q,s),
\]
where $A_r(Q)$ is the corkscrew in the definition of $NTA-$domain.
\begin{lemma}
\label{NonVa}
Let $\left( Q,s\right) \in \partial _{p}D_{T}$ and $r>0$
sufficiently small, depending on $r_{0}$. Then, there exists a constant $C>0$%
, depending on $\mathcal{L}$, $M$, and $r_{0}$ such that 
\begin{equation}
\inf_{C_{r}\left( Q,s\right) \cap D_{T}}\omega ^{\left( x,t\right)
}\left( \Delta _{2r}\left( Q,s\right) \right) \geq C.
\end{equation}
\end{lemma}

The next result states that the  $L-$caloric measure has the so-called doubling property.

\begin{thm}
\label{Doubling-Cal}
There exist a positive constant $C=(X,M,r_{0},\text{diam}$ $\Omega ,T)$
such that for all $(Q,s)\in S$ and $0<r\leq \frac{1}{2}\min
\left\{ r_{0},\sqrt{T-s},\sqrt{s}\right\} $ we have  
\begin{equation}
\omega ^{(x,t)}(\Delta _{2r}(Q,s))\leq C\ \omega ^{(x,t)}(\Delta _{r}(Q,s)),
\end{equation}%
with $d\left( x,Q\right) \leq K\left\vert t-s\right\vert ^{1/2}$ and $
\left\vert t-s\right\vert \geq 16r^{2}.$
\end{thm}

\begin{thm}
\label{Globalcom}
Let $r < \min\left\{r_0/2,\sqrt{(T-s)/4},\sqrt{s/4}\right\}$, $(Q, s)\in S_T$. Let $u$ and
$v$ be two non-negative solutions of $Lu = 0$ in $D_T$, and assume that $u$ and $v$ vanish continuously on $\pa_p D_T\setminus\Delta_{r/2}(Q,s)$. Then, there exists a constant $C=\left(X,M,r_0, \operatorname{diam}(\Omega), T\right)$ such that
\[
u(x_0,T)v(\overline{A}_r(Q,s))\leq Cv(x_0,T)u(\overline{A}_{r}(Q,s)),
\]
where $x_0\in\Omega$ is fixed.
\end{thm}

For any $(y,s)\in \pa_{p} D_{T}$ and $\alpha>0$ a nontangential region at $(y,s)$ is defined by 
\[
\Gamma_{\alpha}(y,s)=\left\{(x,t)\in D_{T}\mid d_{p}((x,t),(y,s))<(1+\alpha)d_{p}((x,t),\pa_{p}D_{T})\right\},
\]
where 
\[
d_{p}((x,t),(y,s))=\left(d(x,y)^{2}+\left|t-s\right|\right)^{1/2}.
\]
\begin{thm}
\label{Non-Max}
Let $\Omega\subset\R^{n}$ be an $NTA$ domain, and set $D_T=\Omega\times (0,T)$.  Given a point $(x_{1},t_{1})\in D_T$, let $f\in L^{1}(S_T,d\omega^{(x_{1},t_{1})})$ and define
\[
u(x,t)=\int_{S_T}f(y,s)d\omega^{(x,t)}(y,s), \quad (x,t)\in D_T.
\] 
Then, $u$ is $L$-caloric in $D_T$ and:
\begin{enumerate}
	\item $N_{\alpha}(u)(y,s)\leq CM_{\omega^{(x_{1},t_{1})}}(f)(y,s), (y,s)\in S_T$;
	\item $u$ converges non-tangentially a.e. $(d\omega^{(x_{1},t_{1})})$ to $f$. 
\end{enumerate}
\end{thm}
Theorem \ref{Globalcom} has the following important consequence.
\begin{thm}
\label{Globalcom-K}
Let $D_T=\Omega\times (0,T)$, where $\Omega\subset\Rn$ be a $ADP_X$ domain and $T>0$. Let $K(\cdot;\cdot)$ be the Poisson kernel defined in (\ref{Def-K}). Fix $x_0\in\Omega$. There exists $r_1>0$ depending on $M$ and $r_0$, and a constant $C=C(X,M,r_0,x_0,T,R_0)>0$, such that given $(Q_0,s_0)\in S_T$, for every $0<r<r_1$ one can find a set $E = E_{(Q_0,s_0),(x_0,T),r}\subset\Delta_r(Q_0,s_0)$, with $\left|E\right|_{d\sigma_X\times dt}=0$, for which
\[
K(x_0,T;y,s)\leq CK(\overline{A}_r(Q_0,s_0);y,s)\omega^{(x_0,T)}\left(\Delta_r(Q_0,s_0)\right)
\]
for every $(y,s)\in\Delta_r(Q,s)\setminus E$.
\end{thm}
\begin{proof}
Let $(Q_0,s_0)\in S_T$. For each $(y,t)\in\Delta_r(Q,s)$ and $0<\hat{r}<r/2$ set
\begin{equation}
\label{u-v}
u(x,t)=\omega^{(x,t)}(\Delta_{\hat{r}}(y,s)),\quad v(x,t)=\omega^{(x,t)}(\Delta_{r/2}(Q_0,s_0)).
\end{equation}
The functions $u$ and $v$ are $L-$parabolic in $D_T$ and vanish continuously on $S_T\setminus\Delta_{2r}(Q_0,s_0)$. Theorem \ref{Globalcom} gives
\begin{equation}
\label{Globalcom-1}
\frac{u(x_0,T)}{v(x_0,T)}\leq \frac{u(\overline{A}_r(Q_0,s_0))}{v(\overline{A}_r(Q_0,s_0))}.
\end{equation}
In terms of the definitions of $u$ and $v$ we obtain from (\ref{Globalcom-1}) 
\begin{equation}
\label{Globalcom-2}
\frac{\omega^{(x_0,T)}(\Delta_{\hat{r}}(y,s))}{\omega^{(x_0,T)}(\Delta_{r/2}(Q_0,s_0))}\leq C\frac{\omega^{\overline{A}_r(Q_0,s_0)}(\Delta_{\hat{r}}(y,s))}{\omega^{\overline{A}_r(Q_0,s_0)}(\Delta_{r/2}(Q_0,s_0))}.
\end{equation}
Upon dividing by $\left|\Delta_{\hat{r}}(y,s)\right|_{d\sigma_X\times dt}$ in (\ref{Globalcom-2}), we find
\begin{equation}
\label{Globalcom-3}
\frac{\omega^{(x_0,T)}(\Delta_{\hat{r}}(y,s))}{\left|\Delta_{\hat{r}}(y,s)\right|_{d\sigma_X\times dt}}\leq C\frac{\omega^{\overline{A}_r(Q_0,s_0)}(\Delta_{\hat{r}}(y,s))}{\left|\Delta_{\hat{r}}(y,s)\right|_{d\sigma_X\times dt}}\frac{\omega^{(x_0,T)}(\Delta_{r/2}(Q_0,s_0))}{\omega^{\overline{A}_r(Q_0,s_0)}(\Delta_{r/2}(Q_0,s_0))}.
\end{equation}
Using Theorem \ref{NonVa} in the right-hand side of (\ref{Globalcom-3}), we conclude
\begin{equation}
\label{Globalcom-4}
\frac{\omega^{(x_0,T)}(\Delta_{\hat{r}}(y,s))}{\left|\Delta_{\hat{r}}(y,s)\right|_{d\sigma_X\times dt}}\leq C\frac{\omega^{\overline{A}_r(Q_0,s_0)}(\Delta_{\hat{r}}(y,s))}{\left|\Delta_{\hat{r}}(y,s)\right|_{d\sigma_X\times dt}}\omega^{(x_0,T)}(\Delta_{r}(Q_0,s_0))
\end{equation}
We now observe that (\ref{Doubling-sigmaX}) in Corollary \ref{Basic-ND1} allows to obtain a Vitali covering theorem and differentiate the measure $\omega^{(x_0,T)}$ with respect to the measure $d\sigma_X\times dt$. This means that for $d\sigma_X\times dt$-a.e. $(y,s)\in\Delta_r(Q_0,s_0)$ the 
\[
\lim_{\hat{r} \rightarrow 0}\frac{\omega^{(x_0,T)}(\Delta_{\hat{r}}(y,s))}{\left|\Delta_{\hat{r}}(y,s)\right|_{d\sigma_X\times dt}}\quad\text{exists and equals}\quad \frac{d\omega^{(x_0,T)}}{d\sigma_X\times dt}(y,s).
\]
This being said, passing to the limit as $\hat{r}\rightarrow 0^+$ in (\ref{Globalcom-4}), we obtain for $d\sigma_X\times dt$-a.e. $(y,s)\in\Delta_r(Q_0,s_0)$
\[
\frac{d\omega^{(x_0,T)}}{d\sigma_X\times dt}(y,s)\leq C\frac{d\omega^{\overline{A}_r(Q_0,s_0)}}{d\sigma_X\times dt}(y,s)\omega^{(x_0,T)}(\Delta_{r}(Q_0,s_0)).
\]
Since by (\ref{Omega-sigma}) in Theorem \ref{Main-P-K} we know that 
\[
\frac{d\omega^{(x_0,T)}}{d\sigma_X\times dt}(y,s)=K(x_0,T;y,s),\ \ \ \ \frac{d\omega^{\overline{A}_r(Q_0,s_0)}}{d\sigma_X\times dt}(y,s)=K(\overline{A}_r(Q_0,s_0);y,s),
\]
we have reached the desired conclusion. 
\end{proof}

\begin{proof}[Proof of Theorem \ref{HolderK}]
We fix $p>1$, $(x_0,t_0)\in S_T$ and $x_1\in\Omega$. Let $R_1$ be the minimum of the constants in Definitions \ref{Def-X-ball}, \ref{NTA-def} and in Theorem \ref{Globalcom-K}. Moreover, we choose $R_1$ so small that $d(x_0,x_1)>MR_1$. Let $0<r<R_1$. If $A_{r}(x_0)$ is a corkscrew for $x_0$, then by the definition of a corkscrew, the triangle inequality and (\ref{Q-Qx}), it is easy to see that for all $(y,s)\in\Delta_r(x_0,t_0)$
\begin{equation}
\label{Cork-K}
d(A_r(x_0),y)\sim Cr, \quad \left|t_0+2r^2-s\right|\sim Cr^2,
\end{equation}
and
\begin{equation}
\label{Cork-K1}
 \left|B_d(x_0,r)\right|\leq C\left|B_d(A_r(x_0),\sqrt{t_0+2r^2-s})\right|.
\end{equation}
We now have 
\begin{eqnarray*}
&&\left(\frac{1}{\left|\Delta_r(x_0,t_0)\right|_{d\sigma_X\times dt}}\int_{\Delta_r(x_0,t_0)}K(x_1,T; y, s)^p d\sigma_X(y) ds\right)^{\frac{1}{p}}\quad (\text{by (\ref{Omega-sigma})}) \\
&=&\left(\frac{1}{\left|\Delta_r(x_0,t_0)\right|_{d\sigma_X\times dt}}\int_{\Delta_r(x_0,t_0)}K(x_1,T; y, s)^{p-1} d\omega^{(x_1,T)}(y,s)\right)^{\frac{1}{p}} \quad (\text{by Theorem \ref{Globalcom-K}})\\
&\leq & C \left(\frac{\omega^{(x_1,T)}\left(\Delta_r(x_0,t_0)\right)}{\left|\Delta_r(x_0,t_0)\right|_{d\sigma_X\times dt}}\int_{\Delta_r(x_0,t_0)}K(\overline{A}_r(x_0,t_0); y, s)^{p-1} d\omega^{(x_1,T)}(y,s)\right)^{\frac{1}{p}} \quad (\text{by (\ref{Bound-P-K})})\\
&\leq & C\left(\frac{\omega^{(x_1,T)}\left(\Delta_r(x_0,t_0)\right)}{\left|\Delta_r(x_0,t_0)\right|_{d\sigma_X\times dt}}\int_{\Delta_r(x_0,t_0)}\left|XG(\overline{A}_r(x_0,t_0); y, s)\right|^{p-1} d\omega^{(x_1,T)}(y,s)\right)^{\frac{1}{p}}\quad (\text{by (\ref{HorBound-G}}) \\
&\leq & C\left(\frac{\omega^{(x_1,T)}\left(\Delta_r(x_0,t_0)\right)}{\left|\Delta_r(x_0,t_0)\right|_{d\sigma_X\times dt}}\left(\frac{r}{r^2\left|B_d(x_0,r)\right|}\right)^{p-1}\int_{\Delta_r(x_0,t_0)} d\omega^{(x_1,T)}(y,s)\right)^{\frac{1}{p}}\quad (\text{by $4)$ in Definition \ref{ADP}}) \\
&\leq & \frac{C}{\left|\Delta_r(x_0,t_0)\right|_{d\sigma_X\times dt}}\int_{\Delta_r(x_0,t_0)} d\omega^{(x_1,T)}(y,s)\quad \text{by (\ref{Omega-sigma})}\\
&=&\frac{C}{\left|\Delta_r(x_0,t_0)\right|_{d\sigma_X\times dt}}\int_{\Delta_r(x_0,t_0)} K(x_1,T;y,s)d\sigma_X(y)ds.
\end{eqnarray*}
This concludes the proof of the reverse H\"older inequality.  Regarding the absolute continuity, we already know from (\ref{Omega-sigma}) that $d\omega^{(x_1,T)}$ is absolutely continuous with respect to $d\sigma_X\times dt$. To prove that $d\sigma_X\times dt$ is absolutely continuous with respect to $d\omega^{(x_1,T)}$  we only need to observe that the reverse H\"older inequality for $K$ established above and the doubling property for $d\sigma_X$ from (\ref{Doubling-sigmaX}) in Corollary \ref{Basic-ND1} allows to invoke Lemma 5 in \cite{CF}. 
\end{proof}

We next establish a reverse H\"older inequality for the kernel $P(x,t; y,s)$ defined in \eqref{Def-P}. To state
the main result we modify the class of $ADP_X$ domains in Definition \ref{ADP}. Specifically, we recall the following definition from \cite{CGN4}.
\begin{definition}
\label{sADP}
Given a system $X = \left\{X_1, ...,X_m\right\}$ of smooth vector fields satisfying (\ref{frc}), we
say that a connected bounded open set $\Omega\subset\Rn$ is $\sigma-$\emph{admissible for the Dirichlet problem} (\ref{Dirichlet-prob2}) with
respect to the system $X$, or simply $\sigma-ADP_X$, if:
\begin{enumerate}
	\item  $\Omega$ is of class $C^{\infty}$.
	\item  $\Omega$ is non-tangentially accessible (NTA) with respect to the Carnot-Carath\'eodory metric
associated to the system $X$ (see Definition \ref{NTA-def});
  \item  $\Omega$ satisfies a uniform tangent outer X-ball condition (see Definition \ref{Def-X-ball});
  \item  The horizontal perimeter measure is upper $1$-Ahlfors regular. This means that there exist
$B,R_o > 0$ depending on $X$ and $\Omega$ such that for every $x\in\pa\Omega$ and $0<r<R_o$ one has
\[
\left(\max_{y\in \pa\Omega\cap B_d(x,r)}\right)\sigma\left(\pa\Omega\cap B_d(x,r)\right)\leq B \frac{\left|B_d(x, r)\right|}{r}.
\]
\end{enumerate}
\end{definition}
The proof of the following result can be found in \cite{CGN4}.
\begin{thm}
\label{Doubling-sigma}
Let $\Omega\subset\Rn$ be a $\sigma-ADP_X$ domain. There exists $C,R_1>0$, depending on the $\sigma-APD_X$ parameters of $\Omega$, such that for every $y\in\pa\Omega$ and $0<r<R_1$,
\[
\sigma(\pa\Omega\cap B_d(y,2r))\leq C\sigma(\pa\Omega\cap B_d(y,r)).
\]
\end{thm}

\begin{proof}[Proof of Theorem \ref{HolderP}]
The relevant reverse H\"older inequality for $P(x_1,t_1;\cdot)$ is proved starting
from the second identity $d\omega^{(x_1,t_1)}= P(x_1,t_1;\cdot)d\left(d\sigma\times dt\right)$ in (\ref{Omega-sigma}) and then arguing in a similar fashion as in
the proof of Theorem \ref{HolderK} but using the non-degeneracy estimate in iv) of Definition \ref{sADP}, instead
of the upper 1-Ahlfors assumption in Definition \ref{ADP}. We leave the details to the interested reader.
\end{proof}

\begin{proof}[Proof of Theorem \ref{NonTanK}]
We start by proving that functions $f\in L^{p}\left(S_T, d\sigma_X\times dt\right)$ are resolutive for the Dirichlet problem \ref{Dirichlet-prob2}. Thanks to Theorem \ref{Brelot} it is enough to prove that $f\in L^1\left(S_T, d\omega^{(x_1,t_1)}\right)$ for some fixed $(x_1,t_1)\in D_T$. From (\ref{Omega-sigma}) and Proposition \ref{Positive-P-K} we have
\begin{eqnarray*}
\int_{S_T}\left|f(y,s)\right|d\omega^{(x_1,t_1)}(y,s)&=&\int_{S_T}\left|f(y,s)\right|K(x_1,t_1;y,s)d\sigma_{X}ds \\
&\leq & \left(\int_{S_T}\left|f(y,s)\right|^{p}d\sigma_{X}ds\right)^{\frac{1}{p}}\left(\int_{S_T}K(x_1,t_1;y,s)^{p'}d\sigma_{X}ds\right)^{\frac{1}{p'}}\\
&\leq & C\left(\int_{S_T}\left|f(y,s)\right|^{p}d\sigma_{X}ds\right)^{\frac{1}{p}}.
\end{eqnarray*}
This shows that $L^{p}\left(S_T, d\sigma_X\times dt\right)\subset L^{1}\left(S_T, d\omega^{(x_1,t_1)}\right)$, and therefore, by Theorem \ref{Brelot}, for each $f\in L^{p}\left(S_T, d\sigma_X\times dt\right)$ the generalized solution $H^{D_T}_f$ exists and it is represented by 
\[
H^{D_T}_f(x,t)=\int_{S_T}f(y,s)d\omega^{(x,t)}(y,s).
\] 
At this point we invoke Theorem \ref{Non-Max} and obtain for every $(y,s)\in S_T$
\begin{equation}
\label{NonMax-1}
N_{\alpha}\left(H^{D_T}_f\right)(y,s)\leq C M_{\omega^{(x_1,t_1)}}(f)(y).
\end{equation}
Moreover, $H^{D_T}_f$ converges non-tangentially $d\omega^{(x_1,t_1)}$-a.e. to $f$. By virtue of Theorem \ref{HolderK}, we also have that $H^{D_T}_f$ converges non-tangentially $d\sigma_X\times dt$-a.e. to $f$. To conclude the proof we need to show that there is a constant $C>0$, depending on $1<p<\infty$, $\Omega$ and $X$, such that 
\[
\left\|N_{\alpha}\left(H^{D_T}_f\right)\right\|_{L^{p}\left(S_T, d\sigma_X\times dt\right)}\leq C\left\|f\right\|_{L^{p}\left(S_T, d\sigma_X\times dt\right)},
\]
for every $f\in L^{p}\left(S_T, d\sigma_X\times dt\right)$. We start by proving the following intermediate estimate
\begin{equation}
\label{NonMax-2}
\left\|M_{\omega^{(x_1,t_1)}}\left(f\right)\right\|_{L^{p}\left(S_T, d\sigma_X\times dt\right)}\leq C\left\|f\right\|_{L^{p}\left(S_T, d\sigma_X\times dt\right)},\quad 1<p\leq\infty.
\end{equation}
Since $p>1$, we choose $\beta$ so that $1<\beta<p$ and fix $(x_1,t_1)\in D_T$ as in Theorem \ref{HolderK}. From (\ref{Omega-sigma}) and the reverse H\"older inequality in Theorem \ref{HolderK}, we have
\begin{eqnarray*}
&&\frac{1}{\omega^{(x_1,t_1)}\left(\Delta_r(x_0,t_0)\right)}\int_{\Delta_r(x_0,t_0)}f(y,s)d\omega^{(x_1,t_1)}(y,s)\\
&\leq& \frac{\left(\int_{\Delta_r(x_0,t_0)}\left|f(y,s)\right|^{\beta}d\sigma_X(y)ds\right)^{\frac{1}{\beta}}}{\omega^{(x_1,t_1)}\left(\Delta_r(x_0,t_0)\right)}\left(\int_{\Delta_r(x_0,t_0)}K(x_1,t_1;y,s)^{\beta'}d\sigma_X(y)ds\right)^{\frac{1}{\beta^{\prime}}}\\
&\leq& C\frac{\left\|f\right\|_{L^{\beta}\left(\Delta_r(x_0,t_0)\right)}}{\omega^{(x_1,t_1)}\left(\Delta_r(x_0,t_0)\right)}\left(\frac{1}{\left|\Delta_r(x_0,t_0)\right|^{\frac{1}{\beta}}_{d\sigma_X\times dt}}\int_{\Delta_r(x_0,t_0)}K(x_1,t_1;y,s)d\sigma_X(y)ds\right)\\
&=& C\left(\frac{1}{\left|\Delta_r(x_0,t_0)\right|_{d\sigma_X\times dt}}\int_{\Delta_r(x_0,t_0)}\left|f(y,s)\right|^{\beta}d\sigma_X(y)ds\right)^{\frac{1}{\beta}}.
\end{eqnarray*}
If we now fix $(y,s)\in S_T$ and take the supremum on both sides of the latter inequality by integrating on every surface box $\Delta_r(x_0,t_0)$ containing $(y,s)$, we obtain
\begin{equation}
\label{Non-Max3}
M_{\omega^{(x_1,t_1)}}(f)(y,s)\leq CM_{\sigma_X}(\left|f\right|^{\beta})(y,s)^{\frac{1}{\beta}}.
\end{equation}
By the doubling property (\ref{Doubling-sigmaX}) in Corollary \ref{Basic-ND1} we know that the space $\left(S_T,d_p(x,t;y,s),d\sigma_Xdt\right)$ is a space of homogeneous type. This allows to use the results in \cite{CW} and invoke the continuity in $L^p\left(S_T,d\sigma_Xdt\right)$ of the Hardy-Littlewood maximal function obtaining
\begin{eqnarray*}
\left\|M_{\omega^{(x_1,t_1)}}f\right\|^p_{L^p\left(S_T,d\sigma_Xdt\right)}&\leq & C\left\|M_{\sigma_X}\left(\left|f\right|^{\beta}\right)^{\frac{1}{\beta}}\right\|^p_{L^p\left(S_T,d\sigma_Xdt\right)}\\
&=&\int_{S_T}M_{\sigma_X}\left(\left|f\right|^{\beta}\right)^{\frac{p}{\beta}}d\sigma_Xdt\leq C\int_{S_T}\left|f\right|^{p}d\sigma_Xdt=C\left\|f\right\|^p_{L^p\left(S_T,d\sigma_Xdt\right)},
\end{eqnarray*}
which proves (\ref{NonMax-2}). The conclusion of the theorem follows from (\ref{NonMax-1}) and (\ref{NonMax-2}).
\end{proof}
\begin{proof}[Proof of Theorem \ref{NonTanP}]
If the domain $\Omega$ is a $\sigma-ADP_X$-domain, instead of a $ADP_X$-domain, then using Theorem \ref{HolderP} instead of Theorem \ref{HolderK} we can establish the solvability of the Dirichlet problem for boundary data in $L^p$ with respect to the standard surface measure. Since the proof of the following result is similar to that of Theorem \ref{NonTanK} (except that we use the second identity in \ref{Omega-sigma} and also Theorem \ref{Doubling-sigma}), we leave the details to the interested reader.
\end{proof}

\end{document}